\documentclass[10pt, letterpaper, twoside]{article} 

\usepackage{amsmath, amsthm, amssymb,bbold,mathrsfs} 
\usepackage{graphicx}
\usepackage[titletoc,title]{appendix}
\usepackage{enumitem}
\setlist[itemize]{itemsep=0.0cm}
\usepackage{thmtools}
\usepackage{thm-restate}
\usepackage[numbers]{natbib}

\usepackage{hyperref}

\usepackage{cleveref}

\newtheorem{prop}{Proposition}[section]
\newtheorem{theorem}{Theorem}[section]

\newtheorem{assumption}{Assumption}[section]

\declaretheorem[name=Theorem,numberwithin=section]{thm}
\declaretheorem[name=Lemma,numberwithin=section]{lem}

\theoremstyle{remark}
\newtheorem{rem}{Remark}[section]

\theoremstyle{definition}
\newtheorem{example}{Example}[section]
\newtheorem{definition}{Definition}[section]

\numberwithin{equation}{section}
 \def\argmin{\mathop{\arg\min}}
 \def\A{\mathbb{A}}
 \def\X{\mathbb{X}}
 \def\E{\mathbb{E}}
 
 \def\R{\mathbb{R}}
 \def\An{\mathcal{A}}
 \def\C{\mathcal{C}}
 
 \def\P{\mathcal{P}}
 \def\B{\mathcal{B}}
 \def\M{\mathcal{M}}
 \def\U{\mathcal{U}}
 \def\q{\gamma}
 \def\I{E}
 \def\uh{\underline{h}}
 \def\bh{\bar{h}}

  \def\epi{\mathop{\text{epi}}}
 \def\eto{\overset{e}{\to}}
 \def\cl{\mathop{\text{cl}}}
 \def\elim{\mathop{\text{e-}\lim}}

 \def\0{\mathbf{0}}
 \def\1{\mathbf{1}}
 \def\ind{\mathbb{1}}

\begin{document} 
\markboth{On Average-Cost Borel-Space MDPs}{}

\title{On Markov Decision Processes with Borel Spaces\\ and an Average Cost Criterion\thanks{This research was supported by a grant from Alberta Innovates---Technology Futures.}}

\author{Huizhen Yu\thanks{RLAI Lab, Department of Computing Science, University of Alberta, Canada (\texttt{janey.hzyu@gmail.com})}}
\date{}

\maketitle

\begin{abstract}
We consider average-cost Markov decision processes (MDPs) with Borel state and action spaces and universally measurable policies. 
For the nonnegative cost model and an unbounded cost model, we introduce a set of conditions under which we prove the average cost optimality inequality (ACOI) via the vanishing discount factor approach. Unlike most existing results on the ACOI, which require compactness/continuity conditions on the MDP, our result does not and can be applied to problems with discontinuous dynamics and one-stage costs. The key idea here is to replace the compactness/continuity conditions used in the prior work by what we call majorization type conditions. In particular, among others, we require that for each state, on selected subsets of actions at that state, the state transition stochastic kernel is majorized by finite measures, and we use this majorization property together with Egoroff's theorem to prove the ACOI. 

We also consider the minimum pair approach for average-cost MDPs and apply the majorization idea. For the case of a discrete action space and strictly unbounded costs, we prove the existence of a minimum pair that consists of a stationary policy and an invariant probability measure induced by the policy. This result is derived by combining Lusin's theorem with another majorization condition we introduce, and it can be applied to a class of countable action space MDPs in which, with respect to the state variable, the dynamics and one-stage costs are discontinuous.
\end{abstract}

average cost; optimality inequality; minimum pair; majorization conditions

\clearpage
\tableofcontents

\clearpage
\section{Introduction}

We consider discrete-time Markov decision processes (MDPs) with Borel state and action spaces, for the average cost criterion where the objective is to minimize the (limsup) expected long-run average cost per unit time.
Specifically, we are interested in the universal measurability framework, which involves lower semi-analytic one-stage cost functions and universally measurable policies. It is a mathematical formulation of MDPs developed to resolve measurability difficulties in dynamic programming on Borel spaces (Strauch \cite{Str-negative}; Blackwell~\cite{Blk-borel}; Blackwell, Freedman, and Orkin~\cite{BFO74}; Shreve~\cite{Shr79}; Shreve and Bertsekas \cite{bs,ShrB78,ShrB79}). An in-depth study of this theoretical framework is given in the monograph~\cite[Part~II]{bs}, and optimality properties of finite- and infinite-horizon problems with discounted and undiscounted total cost criteria have been analyzed \cite{bs,MS92,ShrB78,ShrB79}. The average cost problem has not been thoroughly studied in this framework, however, and the primary purpose of this paper is to investigate the subject further.

To progress toward our goal, we will draw heavily from the rich literature on a subclass of Borel-space MDPs that have certain compactness and continuity properties---these properties help remove major measurability-related issues and also lead to strong optimality results. 
Most notably, there has been extensive research on lower semicontinuous models and sophisticated theories for the average cost criterion have been developed. (The literature on this subject is too vast to list in full; see the early work \cite{Sch92,Sch93,Sen89}, an early survey paper~\cite{ArB93}, the books \cite{HL96,HL99,Puterman94,Sen99}, and the recent work \cite{FKZ12,JaN06,VAm15} and the references therein.)

In order to derive analogous average-cost optimality results for general Borel-space MDPs, our main idea is to replace the compactness/continuity conditions used in the prior work for lower semicontinuous models by what we call majorization type conditions. These conditions will have different forms when we employ different methods of analysis. But roughly speaking, we want to have finite measures that majorize the state transition stochastic kernel of the MDP or some sub-stochastic kernel created from that kernel, at certain action sets for each state or at the admissible state-action pairs, depending on the context. Our idea is to use those majorizing finite measures in combination with Egoroff's or Lusin's theorem, which would then allow us to extract arbitrarily large sets (large as measured by a given finite measure) on which certain functions involved in our analyses have desired uniform convergence or continuity properties.   
We use this technique, along with other analysis techniques developed in the prior work, to obtain two main results in this paper that can be applied to certain classes of MDPs with discontinuous dynamics and one-stage costs.

Our first result based on this majorization idea is a proof of the average cost optimality inequality (ACOI) for two types of MDPs, the nonnegative cost model and an unbounded cost model with a Lyapunov-type condition, without using compactness/continuity conditions. The study of ACOI was initiated by Sennott \cite{Sen89}, who proved it for countable-space MDPs; prior to \cite{Sen89}, the ACOE (average cost optimality equation) was the research focus. Cavazos-Cadena's counterexample \cite{CCa91} showed that the ACOI is more general: in a countable-space MDP in the example, the ACOI has a solution and yet the ACOE does not. For Borel-space MDPs under various compactness/continuity conditions, the ACOI was first established by Sch{\"a}l~\cite{Sch93}, whose results have been further extended since then (see e.g., Hern\'{a}ndez-Lerma and Lasserre~\cite{HL96,HL99} and more recently, Vega-Amaya~\cite{VAm03}; Ja\'{s}kiewicz and Nowak \cite{JaN06}; Feinberg, Kasyanov, and Zadoianchuk~\cite{FKZ12}). 

The two MDP models we consider have been studied in some of the references just mentioned. 
As in those studies, to prove the ACOI, we also use the vanishing discount factor approach, which treats the average cost problem as the limiting case of the discounted problems, and we adopt some of the conditions formalized in \cite{FKZ12,HL99,Sch93} regarding the value functions of the discounted problems. 
In place of the compactness/continuity conditions used in the prior work, we introduce a set of new conditions of the majorization type: among others, we require that for each state, on selected subsets of actions at that state, the state transition stochastic kernel is majorized by finite measures (see Assumptions~\ref{cond-uc-2} and~\ref{cond-pc-2}). We use this majorization property together with Egoroff's theorem (which shows pointwise convergence of functions is ``almost'' uniform convergence as measured by a given finite measure) to prove the ACOI (see Theorems~\ref{thm-uc-acoi} and~\ref{thm-pc-acoi}). 

For comparison, let us mention a few early results that are either about or applicable to average-cost MDPs with universally measurable policies. In particular, Gubenko and Shtatland also introduced a majorization condition to prove the ACOE for Borel-space MDPs without compactness/continuity conditions~\cite[Theorem 2$'$]{GuS75} (measurability issues are assumed away in this theorem). However, they pursued a contraction-based fixed point approach and their majorization condition, formed to make the contraction argument work, not only differs in essential ways from ours but is also too stringent to be practical (see Remark~\ref{rem-GuS-maj-cond} for details). Gubenko and Shtatland \cite{GuS75} studied the ACOE also under an alternative, minorization condition using the same fixed point approach, and that type of sufficient condition for the ACOE has been generalized by Kurano~\cite{Kur86} to one of a multistep-contraction type. Dynkin and Yushkevich \cite[Chap.~7.9]{DyY79} and Piunovski~\cite{Piu89} studied the characteristic properties of canonical systems---a general form of the ACOE together with stationary policies that solve or almost solve the ACOE. These early researches on general Borel-space MDPs differ significantly from ours in both the approaches taken and the results obtained.

Our second result based on the majorization idea is about the existence of a minimum pair in average-cost MDPs.
A minimum pair refers to a policy together with an initial state distribution that attains the minimal average cost over all policies and initial state distributions. Of interest is the existence of such a pair with special structures, in particular, a stationary policy with an associated invariant probability measure, for the stationary policy is then not only average-cost optimal for that initial distribution but it is also pathwise optimal under additional recurrence conditions~\cite[Chap.\ 5.7]{HL96}. The minimum pair approach was proposed by Kurano \cite{Kur89}, motivated by the methods of occupancy measures from Borkar~\cite{Bor83,Bor84}. Unlike the vanishing discount factor approach, it is a direct method.
Kurano~\cite{Kur89} considered bounded costs and compact spaces, and Hern\'{a}ndez-Lerma (\cite{HLe93}; see also the book \cite[Chap.\ 5.7]{HL96}) analyzed the case of strictly unbounded costs, both working with lower semicontinuous MDP models.

Our result is for a discrete action space and strictly unbounded costs. We prove the existence of a minimum pair that consists of a stationary policy and an invariant probability measure of the Markov chain it induces, under another majorization condition we introduce (see Assumption~\ref{cond-pc-3} and Theorem~\ref{thm-su-mp}).
The result is derived by combining the majorization property with Lusin's theorem (which is about the continuity of Borel measurable functions when restricted to some arbitrarily ``large'' closed sets, with largeness measured by a given finite measure). It applies to a class of discrete action space MDPs where the dynamics and one-stage costs are discontinuous with respect to (w.r.t.) the state variable. 
It can be compared with the minimum pair results for lower semicontinuous models in \cite{HLe93,HL96,Kur89}, although its scope is limited because with our current proof arguments, we can only handle discrete action spaces. Future work is to extend this result to Borel action spaces and universally measurable policies.

We remark that Lusin's theorem has been used earlier in a similar way by the author to tackle measurability-related issues in policy iteration for a lower semicontinuous, Borel-space MDP model~\cite[Sec.~6]{YuB-mvipi}. The minimum pair problem we address in this paper and the other arguments involved in our analysis are entirely different from those in \cite{YuB-mvipi}, however.

Besides the results mentioned above, in this paper, we also derive a basic average-cost optimality theorem for the nonnegative cost and unbounded cost MDP models mentioned earlier, without extra ACOI- or majorization-related conditions. It shows that the optimal average cost function is lower semi-analytic and there always exists a universally measurable, $\epsilon$-optimal semi-Markov policy (see Theorem~\ref{thm-ac-basic}). Based on known counterexamples from Dynkin and Yushkevich \cite[Chap.\ 7]{DyY79} and Feinberg~\cite{Fei80}, without additional assumptions on the MDP, this is the strongest conclusion possible (see~Remark~\ref{rmk-basicthm-compare} for details).

The rest of the paper is organized as follows. In Section~\ref{sec-2}, we introduce the universal measurability framework for Borel-space MDPs, and to prepare for subsequent analyses, we derive several basic optimality results for the two models we consider, under the average and discounted cost criteria. In Section~\ref{sec-3}, we consider the vanishing discount factor approach, propose new majorization type conditions, and prove the ACOI for the aforementioned models. In Section~\ref{sec-4}, we consider the minimum pair approach in the case of strictly unbounded costs, and we present our results for discrete action spaces under a new majorization condition we introduce. 
Some background material and proof details are given in Appendix~\ref{appsec-opt}.

\section{Background and Preliminary Analysis} \label{sec-2}

To study general Borel-space MDPs, we need to go beyond Borel measurable functions and policies because there are measurability difficulties otherwise ~\cite{Blk-borel,Str-negative}. 
The universal measurability framework for MDPs is quite involved, however, so before describing it, we need to first introduce several basic definitions and terminologies. We present these introductory materials in Section~\ref{sec-2.1}. They are largely based on the monograph \cite[Part II]{bs} and are similar to the background overview the author gave in \cite{YuB-mvipi}.  

We then present a preliminary analysis of average-cost MDPs in Section~\ref{sec-2.2}, where we define two model classes and derive some basic optimality results for them (Theorems~\ref{thm-ac-basic} and \ref{thm-dcoe}). In the subsequent section, we will impose further conditions on the two models in order to derive more special average-cost optimality results.

\subsection{Borel-space MDPs in the Universal Measurability Framework} \label{sec-2.1}

\subsubsection{Definitions for some Sets and Functions} \label{sec-2.1.1}

A \emph{Borel space} is a topological space that is homeomorphic to a Borel subset of some Polish space (i.e., a separable and completely metrizable topological space) \cite[Def.\ 7.7]{bs}. For a Borel space $X$, let $\B(X)$ denote the Borel $\sigma$-algebra and $\P(X)$ the set of probability measures on $\B(X)$. We shall refer to these probability measures as Borel probability measures. We endow the space $\P(X)$ with the topology of weak convergence; then $\P(X)$ is also a Borel space~\cite[Chap.\ 7.4]{bs}. 
Each Borel probability measure $p$ has a unique extension on a larger $\sigma$-algebra $\B_p(X)$, which is the $\sigma$-algebra generated by $\B(X)$ and all the subsets of $X$ with $p$-outer measure $0$. This extension is called the \emph{completion of $p$} (cf.\ \cite[Chap.\ 3.3]{Dud02}). 
The \emph{universal $\sigma$-algebra} on $X$ is defined as $\U(X) : = \cap_{p \in \P(X)} \B_p(X)$. 

If a function is $\U(X)$-measurable, we say it is \emph{universally measurable}. Since $\B(X) \subset \U(X)$, a Borel measurable function is universally measurable. Conversely, a universally measurable function $f$ is measurable w.r.t.\ the completion of any Borel probability measure $p$ since $\U(X) \subset \B_p(X)$, and one implication of this is that the integral $\int f dp$ for a nonnegative $f$ can be defined w.r.t.\ the completion of $p$. This is the definition for integration that will be used for Borel-space MDPs.

Let $X$ and $Y$ be Borel spaces. A \emph{Borel or universally measurable stochastic kernel} on $Y$ given $X$ is a function $q: X \to \P(Y)$, denoted $q(dy \,|\, x)$, such that 
for each $B \in \B(Y)$, the function $q(B \mid \cdot): X \to [0, 1]$ is Borel or universally measurable, respectively. The definition is equivalent to that 
$q$ is a measurable function from the space $(X, \B(X))$ or $(X, \U(X))$, respectively, to the space $(\P(Y), \B(\P(Y)))$; see \cite[Def.\ 7.12, Prop.\ 7.26 and Lemma 7.28]{bs}. If $q$ is a continuous function, we say that the stochastic kernel $q(dy \mid x)$ is \emph{continuous} (also known as \emph{weak Feller} in the literature).

We now introduce analytic sets and lower semi-analytic functions.
Analytic sets in a Polish space have several equivalent definitions, one of which is that they are the images of Borel subsets of some Polish space under continuous or Borel measurable functions (see e.g., \cite[Prop.\ 7.41]{bs}, \cite[Sec.\ 13.2]{Dud02}). More precisely, in a Polish space $Y$,
the empty set is analytic by definition, and
a nonempty set $D$ is \emph{analytic} if $D=f(B)$ for some Borel set $B$ in a Polish space and Borel measurable function $f:B \to Y$ \cite[Thm.\ 13.2.1(c$'$)]{Dud02}.
In a Polish space every Borel set is analytic 
and every analytic set is universally measurable (\cite[Cor.\ 7.42.1]{bs}, \cite[Thm.\ 13.2.6]{Dud02}). 
The $\sigma$-algebra generated by the analytic sets is called the \emph{analytic $\sigma$-algebra} and lies in between the Borel and universal $\sigma$-algebras. Thus functions that are analytically measurable (i.e., measurable w.r.t.\ the analytic $\sigma$-algebra) are also universally measurable.

Lower semi-analytic functions are extended real-valued functions whose lower level sets are analytic. 
Specifically, a function $f: D \to [-\infty, \infty]$ is called \emph{lower semi-analytic} if $D$ is an analytic set and 
for every $a \in \R$, the level set $\{ x \in D \!\mid f(x) \leq a\}$ of $f$ is analytic \cite[Def.\ 7.21]{bs}. An equivalent definition is that the epigraph of $f$, $\{(x, a) \!\mid x \in D, f(x) \leq a, a \in \R\}$, is analytic (cf.\ \cite[p.\ 186]{bs}). 
For comparison, $f$ is \emph{lower semicontinuous} if its epigraph is closed.
Since Borel sets are analytic, every extended real-valued, Borel measurable function on a Borel space is lower semi-analytic;
since analytic sets are universally measurable, every lower semi-analytic function is universally measurable.
 
\subsubsection{Some Properties of Analytic Sets and Lower Semi-analytic Functions} \label{sec-2.1.1b}

Analytic sets and lower semi-analytic functions play instrumental roles in the universal measurability framework for Borel-space MDPs. These sets and functions were chosen to be the foundation for a theoretical MDP model, because they possess many properties that are relevant to and important for stochastic dynamic programming. A full account of these properties is beyond our scope, however. 
For that, we refer the reader to the papers \cite{BFO74,MS92,ShrB78} and the monograph \cite[Chap.\ 7]{bs} (for general properties of analytic sets, see also the books \cite[Appendix~2]{DyY79} and \cite{Par67,Sriv-borel}). Below we will only mention some properties that will be used frequently in this paper. They concern measurable selection theorems and the preservation of analyticity or lower semi-analyticity under various operations.

The class of analytic sets in a Polish space is closed under countable unions and countable intersections, and moreover, Borel preimages of analytic sets are also analytic (\cite[Cor.\ 7.35.2, Prop.\ 7.40]{bs}, \cite[Chap.\ 4]{Sriv-borel}). 
These properties of analytic sets are reflected in the properties of lower semi-analytic functions, whose lower level sets are analytic. 
Specifically, in the statements below, let $D$ be an analytic set, and let $X$ and $Y$ be Borel spaces. Throughout the paper, for arithmetic operations involving extended real numbers, we define
$$\infty - \infty = - \infty + \infty = \infty, \qquad 0\cdot \pm \infty = \pm \infty \cdot 0  =0.$$
The following operations on lower semi-analytic functions result in lower semi-analytic functions (see \cite[Lemma 7.30]{bs}):
\begin{itemize}
\item[(i)] For a sequence of lower semi-analytic functions $f_n: D \to [-\infty, \infty]$, $n \geq 1$, the functions $\inf_n f_n$, $\sup_n f_n$, $\liminf_{n \to \infty} f_n$, and $\limsup_{n \to \infty} f_n$ are also lower semi-analytic. 
\item[(ii)] If $g:X \to Y$ is Borel measurable and $f: g(X) \to [-\infty, \infty]$ is lower semi-analytic, then the composition $f \circ g$ is lower semi-analytic.
\item[(iii)] If $f, g: D \to [-\infty, \infty]$ are lower semi-analytic functions, then $f+g$ is lower semi-analytic. In addition, if $f, g \geq 0$ or if $g$ is Borel measurable and $g \geq 0$, then $f g$ is lower semi-analytic. 
\end{itemize}

Another operation on lower semi-analytic functions is integration w.r.t.\ a stochastic kernel.
If $f: X \times Y \to [0, \infty]$ is lower semi-analytic and $q(dy \!\mid x)$ is a Borel measurable stochastic kernel on $Y$ given $X$, then the integral 
$$ \phi(x) = \int_Y f(x, y) \, q(dy \!\mid x) $$
is a lower semi-analytic function on $X$~\cite[Prop.\ 7.48]{bs}. 
(If $q(dy \!\mid x)$ is analytically or universally measurable instead, then $\phi$ is universally measurable~\cite[Prop.\ 7.46 and Sec.\ 11.2]{bs} but \emph{not} necessarily lower semi-analytic.)
The preceding properties are closely related to the structure of the optimal cost functions and the selection of measurable policies in the MDP context.
  
The next two properties concern analytic sets in product spaces or lower semi-analytic functions involving two variables.  
The first property is closely related to the validity of value iteration as well as the structure of the optimal cost function in the MDP context.
If $D$ is an analytic set in $X \times Y$, 
then the projection of $D$ on $X$, $\text{proj}_X(D) = \{ x \!\mid (x,y) \in D \ \text{for some} \ y \}$, is analytic \cite[Prop.~7.39]{bs}.
When applied to level sets of functions, an implication of this is that if $D \subset X \times Y$ is analytic and $f: D \to [-\infty, \infty]$ is lower semi-analytic, then after partial minimization of $f$ over the vertical sections $D_x$ of $D$ for each $x$, the resulting function $f^*: \text{proj}_X(D) \to [-\infty, \infty]$ given by
\begin{equation} \label{eq-minf}
   f^*(x) = \inf_{y \in D_x} f(x, y), \quad \text{where} \ D_x = \{ y \mid (x, y) \in D \}, 
\end{equation}   
is also lower semi-analytic \cite[Prop.\ 7.47]{bs}.
  
The Jankov-von Neumann measurable selection theorem asserts that if $D$ is an analytic set in $X \times Y$, 
then there exists an analytically measurable function $\phi: \text{proj}_X(D) \to Y$ such that the graph of $\phi$ lies in $D$, i.e., $(x, \phi(x)) \in D$ for all $x \in \text{proj}_X(D)$~\cite[Prop.\ 7.49]{bs}. For minimization problems of the form~(\ref{eq-minf}), the theorem is applied to the level sets or epigraphs of lower semi-analytic functions. Together with other properties,
it yields, for each $\epsilon > 0$, the existence of an analytically measurable $\epsilon$-minimizer, as well as the existence of a universally measurable $\epsilon$-minimizer $\phi(\cdot)$ that attains the minimum $f^*(x)$ at every $x \in \text{proj}_X(D)$ where this is possible:
\footnote{Here $\epsilon$ is a constant. The result~\cite[Prop.\ 7.50]{bs}, however, extends to a more general case where the required degree of optimality is different for each $x$. Specifically, given a pair of strictly positive, real-valued functions $\epsilon(\cdot)$ and $\ell(\cdot)$, both assumed to be analytically or universally measurable, there exists a universally measurable function $\phi(\cdot)$ that satisfies (\ref{eq-um-minimizer-1})-(\ref{eq-um-minimizer-2}) with $\epsilon(x)$ and $\ell(x)$ in place of $\epsilon$ and $1/\epsilon$, respectively, in (\ref{eq-um-minimizer-2}). Such a function $\phi(\cdot)$ can be constructed as follows: For $n \geq 1$, let $\phi_n$ denote the function that satisfies (\ref{eq-um-minimizer-1})-(\ref{eq-um-minimizer-2}) with $\epsilon = 1/n$. Let $E_n = \big\{ x \in \text{proj}_X(D) \mid \epsilon(x) > 1/n, \, f^*(x) >- \infty, \, \argmin_{y \in D_x} f(x,y) = \emptyset \big\}$ and let $F_n =  \big\{ x \in \text{proj}_X(D)  \mid \ell(x) < n, \, f^*(x) = - \infty, \, \argmin_{y \in D_x} f(x,y) = \emptyset \big\}$. Clearly, $\cup_{n \geq 1} (E_n \cup F_n) = \big\{ x \in \text{proj}_X(D) \mid \argmin_{y \in D_x} f(x,y) = \emptyset\big\}$, which is a universally measurable set \cite[Prop.~7.50(b)]{bs}. The sets $E_n, F_n$ are also universally measurable. Then define $\phi(x) = \phi_n(x)$ on the set $(E_n \cup F_n) \setminus \cup_{k < n} (E_k \cup F_k)$ for $n \geq 1$; and on the set $\big\{ x \in \text{proj}_X(D) \mid  \argmin_{y \in D_x} f(x,y) \not= \emptyset \big\}$, let $\phi(x) = \phi_1(x)$. This function $\phi(\cdot)$ is universally measurable and satisfies
(\ref{eq-um-minimizer-1})-(\ref{eq-um-minimizer-2}) with $\epsilon(x)$ and $\ell(x)$ in place of $\epsilon$ and $1/\epsilon$, respectively, as required.
\label{footnote-ext-selection}}
\begin{equation} \label{eq-um-minimizer-1}
    \phi(x) \in \argmin_{y \in D_x} f(x,y), \qquad \text{if} \  \argmin_{y \in D_x} f(x,y) \not= \emptyset,
\end{equation}
and for $x \in \text{proj}_X(D)$ with $\argmin_{y \in D_x} f(x,y) = \emptyset$,
\begin{equation}   \label{eq-um-minimizer-2}
 \phi(x) \in D_x \qquad \text{and} \qquad  f(x, \phi(x)) \leq  \begin{cases}
        f^*(x) + \epsilon, & \text{if} \ f^*(x) > - \infty; \\
        -1/\epsilon, & \text{if} \ f^*(x) = - \infty.
        \end{cases}
\end{equation}
Further details about these measurable selection theorems can be found in~\cite[Prop.\ 7.50]{bs}. 
For MDPs, this is closely related to the existence of optimal or nearly optimal policies and their structures.

\subsubsection{Definitions for Borel-space MDPs} \label{sec-2.1.2}

In the universal measurability framework, a Borel-space MDP has the following elements and model assumptions (cf.\ \cite[Chap.\ 8.1]{bs}):
\begin{itemize}
\item The state space $\X$ and the action space $\A$ are \emph{Borel spaces}.
\item The control constraint is specified by a set-valued map $A: x \mapsto A(x)$, where for each state $x \in \X$, $A(x) \subset \A$ is a nonempty set of admissible actions at that state, and the graph of $A(\cdot)$,
$$\Gamma = \{(x, a) \mid x \in \X, a \in A(x)\} \subset \X \times \A,$$
is \emph{analytic}.
\item The one-stage cost function $c: \Gamma \to [-\infty, +\infty]$ is \emph{lower semi-analytic}.
\item State transitions are governed by $q(dy \mid x, a)$, a \emph{Borel measurable} stochastic kernel on $\X$ given $\X \times \A$.
\end{itemize}

We consider infinite horizon control problems. A policy consists of a sequence of stochastic kernels on $\A$ that specify for each stage, which admissible actions to apply, given the history up to that stage. 
In particular, a \emph{universally measurable policy} is a sequence $\pi=(\mu_0, \mu_1, \ldots)$, where for each $k \geq 0$,
$\mu_k\big(da_k \!\mid x_0, a_0, \ldots, a_{k-1}, x_k \big)$ is a universally measurable stochastic kernel on $\A$ given $(\X \times \A)^{k} \times \X$ and obeys the control constraint of the MDP:
\footnote{In (\ref{eq-control-constraint}), the probability of the set $A(x_k)$ is measured w.r.t.\ the completion of the Borel probability measure $\mu_k(da_k \mid x_0, a_0, \ldots, a_{k-1}, x_k )$. This is valid because for each $x \in \X$, the vertical section $A(x)$ of the analytic set $\Gamma$ is universally measurable by \cite[Lemma 7.29]{bs}.}
\begin{equation}  \label{eq-control-constraint}
   \mu_k\big(A(x_k) \!\mid x_0, a_0, \ldots, a_{k-1}, x_k \big) = 1, \qquad \forall \, (x_0, a_0, \ldots, a_{k-1}, x_k) \in (\X \times \A)^k \times \X.
\end{equation}   
A policy $\pi$ is \emph{Borel measurable} 
if each component $\mu_k$ is a Borel measurable 
stochastic kernel; $\pi$ is then also universally measurable by definition.
(A Borel measurable policy, however, may not exist \cite{Blk-borel}.)
We define the policy space $\Pi$ of the MDP to be the set of universally measurable policies. We shall simply refer to these policies as policies, dropping the term ``universally measurable,'' if there is no confusion or no need to emphasize their measurability.

We define several subclasses of policies in the standard way: 
A policy $\pi$ is \emph{nonrandomized} if for every $k \geq 0$ and every $(x_0, a_0, \ldots, a_{k-1}, x_k)$, $\mu_k\big(d a_k \!\mid x_0, a_0, \ldots, a_{k-1}, x_k \big)$ is a Dirac measure that assigns probability one to a single action in $A(x_k)$.
A policy $\pi$ is \emph{semi-Markov} if for every $k \geq 0$, the function $(x_0, a_0, \ldots, a_{k-1}, x_k) \mapsto  \mu_k(d a_k \!\mid x_0,  a_0, \ldots, a_{k-1}, x_k)$ depends only on $(x_0, x_k)$;
\emph{Markov} if for every $k \geq 0$, that function depends only on $x_k$; \emph{stationary} if $\pi$ is Markov and $\mu_k = \mu$ for all $k \geq 0$. For the stationary case, we simply write $\mu$ for $\pi = (\mu, \mu, \ldots)$. 
A nonrandomized stationary policy $\mu$ can also be viewed as a function that maps each $x \in \X$ to an action in $A(x)$.  
We denote this mapping also by $\mu$ and we will use both notations $\mu(x)$, $\mu(d a\,|\, x)$ in the paper. 

Because the graph $\Gamma$ of the control constraint $A(\cdot)$ is analytic, by the Jankov-von Neumann selection theorem \cite[Prop.~7.49]{bs}, there exists at least one universally measurable, 
nonrandomized stationary policy. Thus \emph{the policy space $\Pi$ is non-empty}.
Given a policy $\pi \in \Pi$ and an initial state distribution $p_0 \in \P(\X)$, the collection of stochastic kernels 
\begin{align*}
 & \mu_0(d a_0 \mid x_0), \  q(dx_1 \mid x_0, a_0), \  \mu_1(da_1, \mid x_0, a_0, x_1), \ q(dx_2 \mid x_1, a_1), \ \ldots, \qquad \\
 & \ldots, \ \mu_k\big(d a_k \mid x_0, a_0, \ldots, a_{k-1}, x_k \big), \ q(d x_{k+1} \mid x_k, a_k), \ \ldots,
\end{align*}  
determines uniquely a probability measure $r(\pi, p_0)$ on the universal $\sigma$-algebra on $(\X \times \A)^\infty$ \cite[Prop.\ 7.45]{bs}.
\footnote{Because the universal $\sigma$-algebra on $(\X \times \A)^\infty$ is not a product $\sigma$-algebra, the existence of a unique probability measure $r(\pi, p_0)$ here does not follow immediately from the Ionescu Tulcea theorem.}
Furthermore, by \cite[Prop.\ 7.45]{bs}, w.r.t.\ this probability measure, 
the expectation $\E f$ for any nonnegative, universally measurable function $f: (\X \times \A)^{k+1} \to [0, \infty]$ equals the iterated integral 
$$ \int_\X \int_\A \cdots \int_{\X} \int_{\A}  f(x_0, a_0, \ldots, x_{k}, a_{k})  \mu_k(da_{k} \!\mid x_0, a_0, \ldots, x_k) \, q(dx_k \!\mid x_{k-1}, a_{k-1}) \, \cdots \,\mu_0(da_0 \!\mid x_0) \, p_0(dx_0).$$
(Recall that whenever a Borel probability measure appears in the integral of a universally measurable function, the integration is defined w.r.t.\ the completion of the Borel probability measure.)
In general, for a universally measurable function $f: (\X \times \A)^\infty \to [-\infty, +\infty]$, define $\E f : = \E f^+ - \E f^-$ where $f^+ = \max \{ 0, f \}$ and $f^- = - \min \{ 0, f\}$; if $\E f^+ = \E f^- = +\infty$, define $\E f = +\infty$ by following the convention $\infty - \infty =  - \infty + \infty =  \infty.$
In the control problems that we will study, however, we will not encounter such summations.

\subsubsection{The Expected Average Cost and Discounted Cost Criteria}

We consider the average cost criterion and the discounted cost criterion. 
The \emph{$n$-stage value function} of a policy $\pi$ is given by
$$ J_n(\pi, x) : = \E^\pi_x \Big[ \, \textstyle{\sum_{k=0}^{n-1} c(x_k, a_k)} \, \Big], \qquad x \in \X,$$
where $\E^\pi_x$ denotes expectation w.r.t.\ the probability measure induced by $\pi$ and the initial state $x_0 = x$ (cf.\ the explanation given in Section~\ref{sec-2.1.2}). By \cite[Prop.\ 7.46]{bs}, the function $J_n(\pi, \cdot)$ is universally measurable.
We define the \emph{average cost function} of $\pi$ by
$$ J(\pi,x) : = \limsup_{n \to \infty} J_n(\pi, x) /n, \qquad x \in \X,$$
and the \emph{optimal average cost function} by
$$ g^*(x) : = \inf_{\pi \in \Pi} J(\pi,x) = \inf_{\pi \in \Pi} \limsup_{n \to \infty} J_n(\pi, x)/n, \qquad x \in \X.$$
For the discounted cost criterion, with a discount factor $0 < \alpha < 1$, we define the \emph{$\alpha$-discounted value function} of a policy $\pi$ by
$$ v^\pi_\alpha(x) : = \limsup_{n \to \infty}  \E^\pi_x \Big[ \, \textstyle{ \sum_{k=0}^{n-1} \alpha^k c(x_k, a_k) } \, \Big], \qquad x \in \X,$$
and the \emph{optimal $\alpha$-discounted value function} by
$$ v_\alpha(x) : = \inf_{\pi \in \Pi} v^\pi_\alpha(x) , \qquad x \in \X.$$

Both $J(\pi, \cdot)$ and $v^\pi_\alpha(\cdot)$ are universally measurable (the latter by \cite[Prop.\ 7.46]{bs}). 
However, whether the optimal cost functions $g^*$ and $v_\alpha$ are universally measurable cannot be deduced immediately from their definitions. It will be shown in the next subsection, for two classes of MDP models, that $g^*$ and $v_\alpha$ are indeed lower semi-analytic functions. This analysis relies on various properties of lower semi-analytic functions and a deep connection between a certain subset of the policy space and an analytic set, which we will explain more in the next subsection and in Appendix~\ref{appsec-dm}.

We now introduce several classes of functions and the dynamic programming operators for an MDP, which will be needed in the subsequent analysis.
Let $\M(\X)$ denote the set of extended real-valued, universally measurable functions on $\X$, and $\M_b(\X)$ the subset of bounded functions in $\M(\X)$. We shall also consider certain subsets of unbounded functions in $\M(\X)$. For a universally measurable function $w : \X \to (0, +\infty)$, which we shall refer to as a weight function,
let 
$$\M_w(\X) : = \big\{ f \mid  \| f \|_w < \infty, f \in \M(\X) \big\}, \qquad \text{where} \  \ \| f\|_w : = \sup_{x \in \X} \big| f(x) \big| /w(x).$$
The space $\M_w(\X)$ endowed with the weighted norm $\|\cdot\|_w$ and the space $\M_b(\X)$ with the supreme norm $\| \cdot\|_\infty$ are both Banach spaces. 

Let $\An(\X)$ denote the set of extended real-valued, lower semi-analytic functions on $\X$. 
Note that $\An(\X) \cap \M_b(\X)$ and $\An(\X) \cap \M_w(\X)$ are closed subsets of $\M_b(\X)$ and $\M_w(\X)$, respectively.
\footnote{This is because convergence in the $\| \cdot\|_\infty$ or $\|\cdot\|_w$ norm implies pointwise convergence, and the pointwise limit of a sequence of lower semi-analytic functions is lower semi-analytic~\cite[Lemma 7.30(2)]{bs}.}
For $0 < \alpha < 1$, define an operator $T_\alpha$ that maps $v \in \M(\X)$ to a function on $\X$ according to 
$$(T_\alpha v)(x) : = \inf_{a \in A(x)} \left\{ c(x, a) + \alpha \int_\X v(y) \, q(dy \mid x, a) \right\}, \qquad x \in \X.$$
For $\alpha = 1$, define an operator $T$ likewise. We shall refer to them as \emph{dynamic programming operators}.

\begin{lem}[cf.\ {\cite[Chap.\ 7]{bs}}] \label{lem-T-lsa}
The operators $T$ and $T_\alpha$, $0 < \alpha < 1,$ map $\An(\X)$ into $\An(\X)$.
\end{lem}

This lemma follows from the model assumptions in the universal measurability framework for MDPs and the properties of analytic sets and lower semi-analytic functions given in Section~\ref{sec-2.1.1b}.
\footnote{Details: Since the state transition stochastic kernel $q(dy \,|\, x, a)$ is Borel measurable, by \cite[Prop.~7.48]{bs}, the integral $\int v(y) \, q(dy \,|\, x, a)$ for a function $v \in \An(\X)$ is a lower semi-analytic function in $(x, a)$. Since $\alpha \geq 0$, by \cite[Lemma~7.30(4)]{bs}, the integral multiplied by $\alpha$ remains to be lower semi-analytic in $(x, a)$. Then, as the one-stage cost function $c(\cdot)$ is lower semi-analytic, the sum $c(x, a) + \alpha \int v(y) \, q(dy \,|\, x, a)$ is lower semi-analytic in $(x, a)$ by~\cite[Lemma~7.30(4)]{bs}. This shows that $T_\alpha v$ and $T v$ are the result of partial minimization of a lower semi-analytic function on the analytic set $\Gamma$ (the graph of the control constraint). So by~\cite[Prop.~7.47]{bs}, $T_\alpha v$ and $T v$ are lower semi-analytic.}

\subsection{Two Model Classes and some Basic Optimality Properties} \label{sec-2.2}

We consider two model classes which we designate as (PC) and (UC):
\begin{itemize}
\item (PC) is simply the nonnegative model where $c \geq 0$. For the average-cost or discounted problem, it is equivalent to the case where $c$ is bounded from below. 
\item In (UC), the one-stage cost function $c$ can be unbounded below or above, but it needs to satisfy a growth condition and moreover, there is a Lyapunov-type condition on the dynamics of the MDP. The precise definition is as follows.
\end{itemize}

\begin{definition}[the model (UC)] \label{def-uc}
There exist a universally measurable weight function $w(\cdot) \geq 1$ and constants $b, \hat c \geq 0$ and $\lambda \in [0,1)$ such that for all $x \in \X$,
\begin{enumerate}
\item[(a)] $\sup_{a \in A(x)} |c(x, a)| \leq \hat c \, w(x)$;
\item[(b)] $\sup_{a \in A(x)} \int_\X w(y) \, q(dy \mid x, a) \leq \lambda w(x) + b$.
\end{enumerate}
\end{definition}

For (UC), the conditions (a)-(b) in its definition ensure that the average cost function of any policy $\pi$ satisfies $\|J(\pi, \cdot) \|_w \leq \ell$ for the constant $\ell =  \hat c \,b /(1-\lambda)$, and hence the optimal average cost function also satisfies $\|g^*\|_w \leq \ell$ and in particular, $g^*$ is finite everywhere. For (PC), $g^* \geq 0$ and it is possible that at some state $x$, $g^*(x) = +\infty$. This possibility will be eliminated later under further assumptions on the model.

The nonnegative model (PC) has been analyzed in \cite[Part II]{bs} under the expected total cost criterion. This book also discusses the expected discounted cost criterion and analyzes a model with bounded costs. It does not address the average cost criterion; nonetheless, some part of its analysis can be applied to the average cost case. In particular, the relations between a Borel-space MDP and a corresponding deterministic control model (DM) defined on spaces of probability measures (\cite{ShrB79} and \cite[Chaps.\ 9.2-9.3]{bs}), by which many optimality results for the total or discounted cost criterion are derived in the book, give us a starting point to study the average cost case in the universal measurability framework. 

The optimality properties stated in Theorem~\ref{thm-ac-basic} below follow from those arguments from the book \cite[Part II]{bs}. 
In this theorem as well as in what follows, by an $\epsilon$-optimal or optimal policy, we mean a policy that is $\epsilon$-optimal or optimal \emph{for all initial states}. If a policy is only $\epsilon$-optimal or optimal for a certain initial state or initial state distribution, we will state that explicitly.

\begin{theorem}[some average-cost optimality properties] \label{thm-ac-basic}
{\rm (PC)(UC)} \ 
\begin{enumerate}
\item[\rm (i)] The optimal average cost function $g^*$ is lower semi-analytic. 
\item[\rm (ii)] For each $\epsilon > 0$, there exists a universally measurable, $\epsilon$-optimal, randomized semi-Markov policy. If there exists an optimal policy for each state $x \in \X$, then there exists a universally measurable, optimal, randomized semi-Markov policy.
\end{enumerate}
\end{theorem}

We give the proof details in Appendix~\ref{appsec-opt}. Specifically, we explain the corresponding deterministic control model (DM) in Section~\ref{appsec-dm}, which will also be needed later in two other proofs, and we then prove Theorem~\ref{thm-ac-basic} in Section~\ref{sec-prf-thmacbasic}. Here let us make a few remarks about this theorem and its proof.

\begin{rem}[comparison with some prior results] \label{rmk-basicthm-compare}
It is known that even for MDPs with a countable state space, a finite action space, and bounded one-stage costs, there need not exist an $\epsilon$-optimal nonrandomized semi-Markov policy \cite[Example 3, Chap.~7]{DyY79} nor an $\epsilon$-optimal randomized Markov policy \cite[Sec.~5]{Fei80}. 
In both of these counterexamples, there exists an optimal policy for each state. So without extra conditions on the MDP, Theorem~\ref{thm-ac-basic}(ii) is the strongest possible.

It is pointed out by Feinberg~\cite{Fei80} that Strauch's results \cite[Lemma 4.1 and the proof of Theorem 8.1]{Str-negative} can be applied to the average cost case where the one-stage costs are bounded below, and they yield, for any $p \in \P(\X)$ and $\epsilon > 0$, the existence of a randomized semi-Markov policy that is $\epsilon$-optimal \emph{$p$-almost everywhere}. The $p$-almost-everywhere optimality here is due to the restriction to only Borel measurable policies. With universally measurable policies, there exist policies that are optimal or nearly optimal \emph{everywhere}. This is true for the finite-horizon and infinite-horizon total cost problems \cite{ShrB78,ShrB79} and also true for the average cost problem, as reflected by Theorem~\ref{thm-ac-basic}(ii).

The $\epsilon$-optimality mentioned above involves a constant $\epsilon$. It can be generalized to a strictly positive function $\epsilon(\cdot)$; such notions of optimality have been considered by Feinberg \cite[Sec.~2.2]{Fei82}. Theorem~\ref{thm-ac-basic}(ii) holds as well with $\epsilon(\cdot)$ in place of a constant $\epsilon$. This can be shown by using the above version of Theorem~\ref{thm-ac-basic}(ii) to construct another policy with the desired $\epsilon(\cdot)$-optimality (the construction is similar to that given in Footnote~\ref{footnote-ext-selection}). Alternatively, one can make a slight change in the proof of Theorem~\ref{thm-ac-basic} to handle the function $\epsilon(\cdot)$ directly, by using the implication/extension of a measurable selection theorem mentioned in Footnote~\ref{footnote-ext-selection}.\qed
\end{rem}

\begin{rem}[about the proof of Theorem~\ref{thm-ac-basic} and the role of (DM)] \label{rmk-basic-thm-2}
To prove the second part of the theorem, we construct the desired $\epsilon$-optimal (or optimal) policy directly from a universally measurable $\epsilon$-optimal (or optimal) solution of (DM). This differs from the proofs of similar existence results for the total and discounted cost criteria given in \cite[Chap.\ 9.6]{bs}, where the existence of nearly optimal policies is analyzed by relating the dynamic programming operator of the original problem to that in (DM) and by transferring the optimality equations from (DM) to the original problem. In the average cost case, neither (DM) nor the original problem need to admit optimality equations or possess other dynamic-programming type of properties, so our proof cannot rely on such properties. 

The deterministic control model (DM) facilitates greatly the analysis. This is \emph{not} because the average cost problem in (DM) could somehow be solved by dynamic programming, \emph{nor} is it because a deterministic model can help us evade measurability issues in the original problem. (DM) is useful because the structure of its optimization problem permits more readily the application of the theory for analytic sets and lower semi-analytic functions. The optimality properties of (DM) thus obtained can then be transferred to the original problem via their correspondence relations. 

Another comment is that the proofs based on (DM) share similarities with but differ from Strauch's proof of \cite[Theorem 8.1]{Str-negative} mentioned in the preceding remark. The major difference is that in \cite{Str-negative} one deals directly with the set of probability measures on \emph{the trajectory space} induced by all policies, whereas with (DM), one deals with only the set of sequences of marginal probability measures induced by the policies and this set turns out to have nicer properties with regard to measurable selection, as mentioned above. (See \cite[Chap.~9.2]{bs} and Appendix~\ref{appsec-opt} for more details). \qed
\end{rem}

In the next section, we will use the vanishing discount factor approach to prove the ACOI for (PC) and (UC) under additional conditions.
That analysis starts with the optimality equations for the $\alpha$-discounted cost criteria ($\alpha$-DCOE), which are given below:

\begin{theorem}[the $\alpha$-DCOE and existence of $\epsilon$-optimal policies] \label{thm-dcoe}
{\rm (PC)(UC)} \ For $0 < \alpha < 1$, the optimal value function $v_\alpha$ is lower semi-analytic and satisfies the $\alpha$-DCOE $v_\alpha = T_\alpha v_\alpha$, i.e.,
$$  v_\alpha (x) =  \inf_{a \in A(x)} \left\{ c(x, a) + \alpha \int_\X v_\alpha(y) \, q(dy \mid x, a) \right\}, \qquad x \in \X.$$
For (PC), $v_\alpha$ is the smallest solution of the $\alpha$-DCOE in $\An(\X)$, whereas for (UC), $v_\alpha$ is the unique solution of the $\alpha$-DCOE in the space $\An(\X) \cap \M_w(\X)$. Furthermore, in both cases, for each $\epsilon > 0$, there exists a universally measurable, $\epsilon$-optimal, nonrandomized stationary policy. 
\end{theorem}

Under additional compactness and continuity conditions, proofs of the $\alpha$-DCOE for (PC) and (UC) can be found in e.g., the papers~\cite{FKZ12,Sch75} and the books \cite[Chap.\ 5]{HL96}, \cite[Chap.\ 8]{HL99}. In the case here, we will use the results of \cite[Part II]{bs} for general Borel-space MDPs to prove the above theorem. 
Specifically, for (PC), this theorem is implied by the optimality results for nonnegative models~\cite[Props.\ 9.8, 9.10, and 9.19]{bs}.

For (UC), it can be shown (see Lemma~\ref{lem-uc-Ta-contraction} in Appendix~\ref{appsec-lem}) that for some universally measurable weight function $\tilde w \geq w$, the operator $T_\alpha$ is a contraction on the closed subset $\An(\X) \cap \M_{\tilde w}(\X)$ of the Banach space $(\M_{\tilde w}(\X), \| \cdot\|_{\tilde w})$. More precisely, for some $\beta \in (\alpha, 1)$,
\begin{equation} \label{eq-contraction-Ta-uc}
T_\alpha v \in \An(\X) \cap \M_{\tilde w}(\X) \quad \text{and} \quad 
\left\| T_\alpha v - T_\alpha v' \right\|_{\tilde w} \leq \beta \left\| v - v' \right\|_{\tilde w}, \quad \forall \, v, v' \in \An(\X) \cap \M_{\tilde w}(\X).
\end{equation}
We use this contraction property of $T_\alpha$ 
together with the correspondence between the original problem and the deterministic control model (DM) \cite[Chap.\ 9]{bs} to prove Theorem~\ref{thm-dcoe} for (UC). The proof is given in Appendix~\ref{appsec-prf-uc}. It is similar to, but does not follow exactly the one given in \cite[Chap.\ 9]{bs} for bounded one-stage costs; see Remark~\ref{rmk-prf-dcoe-uc} at the end of Appendix~\ref{appsec-prf-uc} for further explanations.

As another preparation for the subsequent analysis, let us state a lemma about an implication of the ACOI on the existence and structure of average-cost optimal or nearly optimal policies. For comparison, recall that what Theorem~\ref{thm-ac-basic} just showed is the existence of an $\epsilon$-optimal, randomized semi-Markov policy, in the general case where $g^*$ need not be constant. The proof of this lemma uses mostly standard arguments and is given in Appendix~\ref{appsec-lem}.

\begin{lem}[a consequence of ACOI] \label{lem-optpol}
Consider the models (PC) and (UC) with the average cost criterion. 
Suppose that the optimal average cost function $g^*$ is constant and finite. 
Suppose also that for some finite-valued $h \in \An(\X)$, with $h \geq 0$ for (PC) and $\|h\|_w < \infty$ for (UC), the ACOI holds: $g^* + h \geq T h$, i.e.,
\begin{equation} \label{eq-acoi-gen}
   g^* + h(x) \geq \inf_{a \in A(x)} \left\{ c(x, a) + \int_\X h(y) \, q(dy \mid x, a) \right\}, \qquad x \in \X.
\end{equation}   
Then there exist an optimal nonrandomized Markov policy and, for each $\epsilon > 0$, an $\epsilon$-optimal nonrandomized stationary policy. If, in addition, the infimum in the right-hand side of the ACOI is attained for every $x \in \X$, then there exists an optimal nonrandomized stationary policy. 
\end{lem}

\section{Main Results: The ACOI} \label{sec-3}

In this section, we place additional conditions on the models (PC) and (UC), under which we study the ACOI via the vanishing discount factor approach. Some of these conditions are standard and from prior work, and some are new conditions that we introduce to replace the compactness and continuity conditions used in the prior work to prove the ACOI for lower semicontinuous models. The arguments for the two models (PC) and (UC) are similar but differ in details, so we will discuss (PC) and (UC) in two separate subsections.

\subsection{The Case of Unbounded Costs (UC)} \label{sec-3.1}

We consider the model (UC) first. Let $\bar x$ be some fixed state and consider the relative value functions of the $\alpha$-discounted problems:
\begin{equation} \label{eq-uc-ha}
h_\alpha(x) : = v_\alpha(x) - v_\alpha(\bar x), \qquad x \in \X.
\end{equation}

\subsubsection{Assumptions} \label{sec-uc-cond}

The first assumption is extracted from the prior work on ACOI:

\begin{assumption} \label{cond-uc-1}
For the model (UC), the set of functions $\{h_\alpha \mid \alpha \in (0,1)\}$ as defined above is bounded in $\M_w(\X)$, i.e., $\sup_{\alpha \in (0,1)} \|h_\alpha\|_w < \infty$.
\end{assumption}

In the prior work, a $w$-geometric ergodicity condition has been used to ensure that the functions $h_\alpha$ have the above boundedness property. Specifically, it is assumed, or ensured through other conditions, that every stationary nonrandomized policy induces a $w$-geometric ergodic Markov chain on the state space (see e.g., \cite[Lemma 10.4.2]{HL99} and \cite[p.\ 498]{JaN06}; see \cite[Chap.~15]{MeT09} for the definition of such Markov chains). This together with the existence of $\epsilon$-optimal nonrandomized stationary policies for each $\epsilon >0$ (cf.\ Theorem~\ref{thm-dcoe}) then guarantees the boundedness of the family $\{h_\alpha \mid \alpha \in (0,1)\}$.
In addition to the $w$-geometric ergodicity condition, compactness and continuity conditions are also involved in deriving the ACOI or ACOE in the prior work \cite{HL99,JaN06}.

To start the analysis of the ACOI, we will need an implication of Assumption~\ref{cond-uc-1} given in the following lemma. Take a sequence $\alpha_n \uparrow 1$ such that for some finite number $\rho^*$,
\begin{equation} \label{eq-rho}
  (1 - \alpha_n) \, v_{\alpha_n}(\bar x) \to \rho^* \quad \text{as} \  n \to \infty.
\end{equation}  
This is possible because the model conditions of (UC) imply that for each $x \in \X$, $(1 - \alpha) \, v_\alpha (x)$ is bounded over $\alpha \in (0,1)$.
Corresponding to the sequence $\{\alpha_n\}$, consider the sequence of functions 
$h_n : = h_{\alpha_n}$. 
Define 
\begin{align}
 \uh  : = \liminf_{n \to \infty} h_n, \qquad  \uh_n : = \inf_{m \geq n} h_m, \ \ \ n \geq 0,  \label{eq-uh} \\
 \bh  : = \limsup_{n \to \infty} h_n, \qquad  \bh_n : = \sup_{m \geq n} h_m, \ \ \ n \geq 0. \label{eq-bh}
\end{align}
Note that as $n \to \infty$,
$\uh_n \uparrow \uh$ and $\bh_n \downarrow \bh.$
The next lemma about these functions follows directly from Theorem~\ref{thm-dcoe}, \cite[Lemma 7.30(2)]{bs}, and Assumption~\ref{cond-uc-1}.

\begin{lem} \label{lem-h-lsa-uc}
{\rm (UC)} Under Assumption~\ref{cond-uc-1}, 
all the functions $\uh$, $\bh$, $\uh_n$, $\bh_n$, $n \geq 0$, are lower semi-analytic and lie in a bounded subset of $\M_w(\X)$.
\end{lem}

We now introduce new conditions, which we use to replace compactness and continuity conditions used in the prior work on the ACOI for (UC):

\begin{assumption} \label{cond-uc-2}
In the model (UC), for each $x \in \X$ and $\epsilon > 0$, the following hold:
\begin{itemize}
\item[\rm (i)] There exist a compact set $K \subset \A$ and $0 < \bar \alpha  < 1$ such that for all $\alpha \in [\bar \alpha, 1)$,
\begin{equation} \label{cond-uc-2a}
 \inf_{a \in K \cap A(x)} \left\{ c(x, a) + \alpha \int_\X v_\alpha(y) \, q(dy \mid x, a) \right\} \leq v_\alpha(x) + \epsilon.  
\end{equation} 
\item[\rm (ii)] There exists a (nonnegative) finite measure $\nu$ on $\B(\X)$ that majorizes every $q(dy \,|\, x, a)$, $a \in K \cap A(x)$:
\begin{equation}
 \sup_{a \in K \cap A(x)} q( B \mid x, a) \leq \nu(B), \qquad \forall \, B \in \B(\X). \label{cond-uc-2b} 
\end{equation}
\item[\rm (iii)] The weight function $w(\cdot)$ is uniformly integrable w.r.t.\ $\{ q(dy \,|\, x, a) \mid a \in K \cap A(x) \}$ in the sense that
\begin{equation}
   \lim_{\ell \to \infty} \sup_{a \in K \cap A(x) }  \int_\X w(y) \, \ind\big[w(y) \geq \ell \, \big] \, q(dy \mid x, a) = 0, \label{cond-uc-2c}
\end{equation}
where $\ind(\cdot)$ denotes the indicator function.
\end{itemize}
\end{assumption}

Let us give here a preliminary discussion about these conditions. We will discuss them further and also give an illustrative example after we prove the ACOI,  since the roles of some of these conditions can be better seen then (see Section~\ref{sec-uc-acoi-discussion}).

First, there are cases where all or some of the conditions in Assumption~\ref{cond-uc-2} hold obviously.
For example, if for each $x$, $A(x)$ is finite, then all three conditions are satisfied by letting $K= A(x)$. (This simple setting is not of primary interest to us, however, because optimality results can be derived directly in this case, without using the proof approach that we are going to take.)
More generally, if $A(x)$ or $\A$ is compact, then (i)~holds trivially for $K = A(x)$ or $\A$. 
If $w(\cdot)$ is bounded from above on the union of the supports of the probability measures $q(dy \mid x, a), a \in K \cap A(x)$, such as in the case where the union is contained in a compact set and $w(\cdot)$ is continuous,  then (iii) is clearly satisfied. 
If $c(\cdot)$ is bounded, then $w(\cdot)$ can be chosen to be constant and (iii) then~holds trivially. 
Note that Assumption~\ref{cond-uc-2} does not require the one-stage cost function $c(\cdot)$ to have any special properties.

Regarding Assumption~\ref{cond-uc-2}(iii), it is implied by the slightly stronger, yet much simpler-looking condition $\int w \, d \nu < \infty$, where $\nu$ is the majorizing finite measure in Assumption~\ref{cond-uc-2}(ii). 
The condition $\int w \, d \nu < \infty$, however, can be inconvenient to verify when the measure $\nu$ is too complicated and a direct evaluation of the integral $\int w \, d \nu$ is impractical. In comparison, verifying the condition (iii) can be straightforward when the weight function $w(\cdot)$ has a simple analytical expression (e.g., when $x \in \R$ and $w(x) = e^x$ or $x^2$). This is why we have Assumption~\ref{cond-uc-2}(iii) as is. Note that the situation is different for the condition~(ii) in which $\nu$ appears, because to verify~(ii), we do not need the exact expression of $\nu$. It suffices that some finite measure with the desired majorization property \emph{exists}. This can be inferred qualitatively, in some cases, from the properties of the state transition stochastic kernel $q(dy\mid x,a)$, without the need for exact calculation. 

Assumption~\ref{cond-uc-2}(ii) is the key condition. Our purpose is to use this majorization condition to handle a certain class of discontinuous models. 
Although lower semicontinuous models cover a large class of problems and are mathematically elegant, discontinuity naturally occurs in physical systems. The behavior of such systems can vary gradually within certain regions of the state space but change abruptly across the boundaries of these regions, depending on which physical mechanisms come into effect.
The class of discontinuous models for which Assumption~\ref{cond-uc-2}(ii) can hold naturally are those where $q(dy \mid x, a)$ is not continuous in $a$ or $(x,a)$, but for all $a \in K \cap A(x)$, $q(dy \mid x, a)$ has a density $f_{x,a}$ w.r.t.\ a common ($\sigma$-finite) reference measure $\varphi$. The pointwise supremum of the density functions, $f_x: = \sup_{a \in K \cap A(x)} f_{x,a}$ (or a measurable function that upper-bounds it), when it belongs to $\mathcal{L}^1(\X, \B(\X), \varphi)$, 
defines a finite measure $\nu$, with $d\nu = f_x d \varphi$, that has the desired majorization property (\ref{cond-uc-2b}). 

Note that Assumption~\ref{cond-uc-2}(ii) need not be satisfied by continuous state transition stochastic kernels. For example, if $A(x)=[0,1]$ and $q(dy \,|\, x, a) = \delta_a$ (the Dirac measure at $a$), there is no finite measure with the desired majorization property.

Regarding Assumption~\ref{cond-uc-2}(i), in some circumstances, the existence of a compact set $K$ with the property (\ref{cond-uc-2a}) is close to being a necessary condition for the ACOI to hold. We shall discuss this condition further after we analyze the ACOI. There we will explain where this condition comes from, in particular, its connection with the theory on epi-convergence of functions (see Section~\ref{sec-uc-acoi-discussion}). Here let us remark that first, our proof of the ACOI will not use the compactness property of $K$, so it suffices that for some subset $K$ of actions, (\ref{cond-uc-2a}) and the rest of the assumptions hold. Second, the subset $K$ introduced in this condition is important in the subsequent Assumption~\ref{cond-uc-2}(ii): It would be too stringent to require a finite measure $\nu$ to majorize $q(dy \mid x, a)$ for all $a \in A(x)$ instead of $a \in K \cap A(x)$. For example, if $A(x) = \R$ and $q(dy \mid x, a)$ is the normal distribution $\mathcal{N}(a, 1)$ on $\R$, no finite measure can majorize these distributions for all $a \in \R$.

\begin{rem}[about the majorization condition in {\cite{GuS75}}] \label{rem-GuS-maj-cond}
Gubenko and Shtatland used a minorization condition and a majorization condition, alternatively, to convert the dynamic programming operator $T$ into a contraction (roughly speaking), thereby proving the ACOE via a contraction-based fixed point approach \cite[Theorem 2 and 2$'$]{GuS75}. Their majorization condition ~\cite[Sec.~3, Condition~(II)]{GuS75} is like a symmetric counterpart of their minorization condition, and it requires that there exists a finite measure $\nu$ on $\B(\X)$ such that
\begin{equation} \label{eq-GuS-cond}
  q(B \mid x, a) \leq \nu(B), \quad \forall \, B \in \B(\X), \ (x, a) \in \Gamma,  \qquad \text{and} \qquad \nu(\X) < 2. \qquad \qquad
\end{equation}  
Note that here the same measure $\nu$ needs to majorize $q(dy \mid x, a)$ for all states and admissible actions, whereas in our Assumption~\ref{cond-uc-2}, $\nu$ can be different for each state. The requirement $\nu(\X) < 2$ (needed for converting $T$ into a contraction) is too stringent and renders their condition (\ref{eq-GuS-cond}) impractical.\qed
\end{rem}

\subsubsection{Optimality Results} \label{sec-uc-acoi-prf}

We now prove the ACOI for (UC) under the assumptions introduced in the preceding subsection. The result and its proof involve the relative value functions $\{h_n\}$, the functions $\uh$, $\{\uh_n\}$, $\bh$, $\{\bh_n\}$, and also the scalar $\rho^* = \lim_{n \to \infty} (1 - \alpha_n) \, v_{\alpha_n} (\bar x)$ that we defined earlier (cf.~(\ref{eq-uc-ha})-(\ref{eq-bh})).

\begin{thm}[the ACOI for (UC)] \label{thm-uc-acoi}
For the (UC) model, under Assumptions~\ref{cond-uc-1}-\ref{cond-uc-2}, the optimal average cost function $g^*(\cdot) = \rho^*$, and with $\uh \in \An(\X) \cap \M_w(\X)$ as given in (\ref{eq-uh}), the pair $(\rho^*, \uh)$ satisfies the ACOI:
\begin{equation} \label{eq-uc-acoi}
  \rho^* + \uh(x) \geq \inf_{a \in A(x)} \left\{ c(x, a) + \int_\X \uh(y) \, q(dy \mid x, a) \right\}, \qquad x \in \X.
\end{equation}  
Hence there exist an optimal nonrandomized Markov policy and for each $\epsilon > 0$, an $\epsilon$-optimal nonrandomized stationary policy. 
\end{thm}

This theorem also implies that $\rho^*$ does not depend on our choice of the sequence $\{\alpha_n\}$ or the state $\bar x$, and
$$\lim_{\alpha \to 1} (1 - \alpha) \, v_\alpha(x) = g^*, \qquad \forall \, x \in \X.$$ 
To prove the theorem, we first prove a lemma.

\begin{lem} \label{lem-uc-2acoi}
{\rm (UC)} Under the assumptions of Theorem~\ref{thm-uc-acoi}, let $K$ be the compact set in Assumption~\ref{cond-uc-2} for a given $x \in \X$ and $\epsilon > 0$. Then
$$ 
\lim_{n \to \infty} \inf_{a \in K \cap A(x)} \left\{ c(x, a) + \alpha_n \int_\X \uh_n(y) \, q(dy \mid x, a) \right\} = \inf_{a \in K \cap A(x)} \left\{ c(x, a) +  \int_\X \uh(y) \, q(dy \mid x, a) \right\}.
$$
\end{lem} 

\begin{proof} 
For the state $x$, $\epsilon > 0$, and the set $K$ given in the lemma, let $\nu$ be the corresponding finite measure on $\B(\X)$ in Assumption~\ref{cond-uc-2}(ii). Recall that $\uh_n \uparrow \uh$ and these functions are finite-valued and universally measurable (Lemma~\ref{lem-h-lsa-uc}). Therefore, 
by Egoroff's Theorem~\cite[Theorem~7.5.1]{Dud02}, for any $\delta > 0$, there exists a universally measurable set $D_\delta \subset \X$ with $\nu(\X \setminus D_\delta) < \delta$ such that on the set $D_\delta$, $\uh_n$ converges to $\uh$ uniformly as $n \to \infty$.
Consequently, for any $\eta > 0$, it holds for all $n$ sufficiently large that
\begin{equation} \label{eq-prf-thm-uc1a}
   \int_{D_\delta} \big(\uh(y) - \uh_n(y) \big)\, q(dy \mid x, a) \leq  \eta, \qquad \forall \, a \in  A(x).
\end{equation}
We now bound the integral of $\uh - \uh_n$ on the complement set $\X \setminus D_\delta$. 
By Lemma~\ref{lem-h-lsa-uc}, for all $n \geq 0$, $\| \uh - \uh_n \|_w \leq \ell$ for some constant $\ell$. So for all $a \in A(x)$,
$$ \int_{\X \setminus D_\delta} \big(\uh(y) - \uh_n(y) \big)\, q(dy \mid x, a) \, \leq  \, \ell  \int_{\X \setminus D_\delta}  w(y) \, q(dy \mid x, a).$$
By the choice of $D_\delta$ and the majorization property of $\nu$ in Assumption~\ref{cond-uc-2}(ii), we have
\begin{equation} \label{eq-prf-thm-uc1b}
 \sup_{a \in K \cap A(x)} q\big(\X \setminus D_\delta \mid x, a\big)  \leq \nu\big(\X \setminus D_\delta\big) < \delta.
\end{equation} 
By an alternative characterization of uniform integrability~\cite[Theorem 10.3.5]{Dud02}, (\ref{eq-prf-thm-uc1b}) together with the uniform integrability assumption in Assumption~\ref{cond-uc-2}(iii) implies that for any given $\eta > 0$, it holds for all $\delta$ sufficiently small that 
$\int_{\X \setminus D_\delta}  w(y) \, q(dy \mid x, a) \leq \eta$ for all $a \in K \cap A(x)$.
Consequently, given $\eta > 0$, by choosing $\delta$ sufficiently small, we can make
\begin{equation} \label{eq-prf-thm-uc1c}
 \sup_{a \in K \cap A(x)} \int_{\X \setminus D_\delta} \big(\uh(y) - \uh_n(y) \big)\, q(dy \mid x, a) \leq \eta.
\end{equation} 
Combining this with (\ref{eq-prf-thm-uc1a}), we obtain that for all $n$ sufficiently large,
$$ \sup_{a \in K \cap A(x)}  \int_\X  \big(\uh(y) - \uh_n(y) \big)\, q(dy \mid x, a) \leq 2 \eta$$
and hence
$$ \inf_{a \in K \cap A(x)} \left\{ c(x, a) + \alpha_n \int_\X \uh_n(y) \, q(dy \mid x, a) \right\} \geq \inf_{a \in K \cap A(x)} \left\{ c(x, a) + \alpha_n \int_\X \uh(y) \, q(dy \mid x, a) \right\} - 2 \eta.$$
Since $\eta$ is arbitrary and $\uh_n \leq \uh$, the lemma follows by letting $n \to \infty$ on both sides of the preceding inequality (and using also the fact that since $\sup_{a \in K \cap A(x)}  \int_\X |\uh(y) | \, q(dy \mid x, a) < \infty$ under Assumption~\ref{cond-uc-1} and the model condition of (UC), $(1 - \alpha_n) \sup_{a \in K \cap A(x)}  \int_\X |\uh(y)| \, q(dy \mid x, a) \to 0$). 
\end{proof}

\begin{proof}[Proof of Theorem~\ref{thm-uc-acoi}]
For each $x \in \X$ and $\epsilon > 0$, by the $\alpha$-DCOE (Theorem~\ref{thm-dcoe}) and Assumption~\ref{cond-uc-2}(i), for all $n$ sufficiently large
\begin{align}
   ( 1 - \alpha_n) \, v_{\alpha_n}(\bar x) + h_n(x) \, & = \, \inf_{a \in A(x)} \left\{ c(x, a) + \alpha_n \int_\X h_n(y) \, q(dy \mid x, a) \right\} \label{eq-prf-thm-uc-oe} \\
   & \geq \,  \inf_{a \in K \cap A(x)} \left\{ c(x, a) + \alpha_n \int_\X h_n(y) \, q(dy \mid x, a) \right\} - \epsilon \notag \\
   & \geq \,  \inf_{a \in K \cap A(x)} \left\{ c(x, a) + \alpha_n \int_\X \uh_n(y) \, q(dy \mid x, a) \right\} - \epsilon, \label{eq-prf-thm-uc0a}
\end{align}   
where the last inequality used the fact $\uh_n \leq h_n$ and that $\uh_n$ is universally measurable (Lemma~\ref{lem-h-lsa-uc}).
Letting $n \to \infty$ in both sides of (\ref{eq-prf-thm-uc0a}),
we have 
\begin{align*}
 \rho^* + \uh(x) + \epsilon & \, \geq \,  \liminf_{n \to \infty} \inf_{a \in K \cap A(x)} \left\{ c(x, a) + \alpha_n \int_\X \uh_n(y) \, q(dy \mid x, a) \right\} \\
   & \, = \,  \inf_{a \in K \cap A(x)} \left\{ c(x, a) +  \int_\X \uh(y) \, q(dy \mid x, a) \right\} 
    \, \geq \, \inf_{a \in A(x)} \left\{ c(x, a) +  \int_\X \uh(y) \, q(dy \mid x, a) \right\},
\end{align*}
where the equality follows from Lemma~\ref{lem-uc-2acoi}. Since this holds for every $x \in \X$ and $\epsilon$ is arbitrary, the desired inequality (\ref{eq-uc-acoi}) is proved.

To show $g^*(\cdot) = \rho^*$, as in the analysis in~\cite{JaN06}, it suffices to show that for the pair $(\rho^*, \bh)$, where $\bh = \limsup_{n \to \infty} h_n$ as we recall, the opposite inequality holds:
\begin{equation} \label{eq-prf-thm-uc-2a} 
    \rho^* + \bh(x) \leq \inf_{a \in A(x)} \left\{ c(x, a) + \int_\X \bh(y) \, q(dy \mid x, a) \right\}, \qquad x \in \X.
\end{equation}    
This is because the inequality (\ref{eq-prf-thm-uc-2a}) implies that for all policies $\pi$ and state $x$, the average cost $J(\pi,x) \geq \rho^*$, whereas the inequality (\ref{eq-uc-acoi}) just proved implies the opposite relation $J(\pi,x) \leq \rho^*$. (The proof of the former is standard and similar to that of Lemma~\ref{lem-optpol} given in Appendix~\ref{appsec-lem}, and the proof of the latter is the same as that of  Lemma~\ref{lem-optpol}.)
From the $\alpha$-DCOE (\ref{eq-prf-thm-uc-oe}), we have that for all $a \in A(x)$, 
\begin{equation}
( 1 - \alpha_n) \, v_{\alpha_n}(\bar x) + h_n(x) \, \leq \, c(x, a) + \alpha_n \int_\X h_n(y) \, q(dy \mid x, a),
\end{equation}
so by taking limit supremum of both sides as $n \to \infty$ and using also the fact $\bh_n \geq h_n$ by definition, we have
$$ \rho^* + \bh(x) \leq c(x, a) + \limsup_{n \to \infty} \alpha_n \int_\X \bh_n(y) \, q(dy \mid x, a) = c(x, a) +  \int_\X \bh(y) \, q(dy \mid x, a),$$
where the last inequality follows from the dominated convergence theorem, in view of Lemma~\ref{lem-h-lsa-uc} and the model condition of (UC). 
This proves (\ref{eq-prf-thm-uc-2a}) and hence $g^*(\cdot) = \rho^*$ as discussed earlier. Finally, that $\uh \in \An(\X) \cap \M_w(\X)$ follows from Lemma~\ref{lem-h-lsa-uc}, and the existence of an optimal Markov policy and $\epsilon$-optimal stationary policy follows from the ACOI proved above and Lemma~\ref{lem-optpol}.
\end{proof}

\subsubsection{Further Discussion on Assumption~\ref{cond-uc-2} and an Illustrative Example} \label{sec-uc-acoi-discussion}

In this discussion, to simplify notation, for pointwise limits of functions, we shall abbreviate the expression ``$n \to \infty$'' and write, for instance, ``$\lim_n$'' or ``$\liminf_n$'' instead. We shall assume $c(x,a) = +\infty$ if $a \not\in A(x)$, so that we can write ``$\inf_a f(a)$'' instead of ``$\inf_{a \in A(x)} f(a)$'' when $f$ is of the form $f(\cdot) = c(x, \cdot) + \psi(x, \cdot)$ for a given state $x$ and some function $\psi$.

Let us first explain the origin of Assumption~\ref{cond-uc-2}(i). In introducing this condition, we have been influenced by the theory on epi-convergence of functions on $\R^d$ \cite[Chap.\ 7]{RoW98}. A sequence of extended real-valued functions $\{f_n\}$ on $\R^d$ \emph{converges epigraphically} to a function $f$, denoted $f_n \eto f$, if as $n \to \infty$, $\epi(f_n)$ (the epigraph of $f_n$) converges to $\epi(f)$ in the sense of set convergence.
\footnote{For $E, E_n \subset \R^d$, $n \geq 0$, we say \emph{$\{E_n\}$ converges to $E$} if the following two conditions are met: (i) For every subsequence $\{n_k\}$ and convergent sequence $\{y_{n_k}\}$ with $y_{n_k} \in E_{n_k}$ and $y_{n_k} \to \bar y$ as $k \to \infty$, the limit $\bar y \in E$. (ii) For every $\bar y \in E$, there exist $y_n \in E_n$ for all $n$ sufficiently large, such that $y_n \to \bar y$ as $n \to \infty$. (See \cite[Definition 4.1]{RoW98}.)}
We denote this limit by $\elim_n f_n$. It is, by definition, lower semicontinuous and lies below $\liminf_n f_n$. 
For a nondecreasing sequence $\{f_n\}$, $\elim_n f_n$ always exists and equals $\sup_{n} (\cl f_n)$ \cite[Prop.~7.4(d)]{RoW98}, where $\cl f_n$ is the closure of $f_n$ (i.e., the function whose epigraph equals the closure of $\epi(f_n)$ or in other words, the largest lower semicontinuous function majorized by $f_n$). 
To relate such $f_n$ and their epi-limits to the functions involved in our problem, suppose that for every $x \in \X$, $A(x) \subset \Re^d$. Assume also that $\uh_n = h_n \geq 0$ for all $n$, for simplicity. Let us investigate when the ACOI is impossible. This will show us what kind of condition is needed for the desired ACOI to hold. 

For a given state $x$, let 
$f_n(a) = c(x,a) + \alpha_n \int_\X \uh_n(y) \, q(dy \mid x, a).$
Since the nonnegative functions $\uh_n \uparrow \uh$, the sequence $\{f_n\}$ is nondecreasing and hence $\elim_n f_n$ exists, as discussed earlier. By the monotone convergence theorem, the pointwise limit $\lim_n f_n$ also exists, and it is the function $c(x, \cdot) + \int_\X \uh(y) \, q(dy \mid x, \cdot)$ and lies above the epi-limit. 
Hence for all $m \geq 0$,
\begin{equation} \label{eq-diss-uc-cond-1}
\textstyle{\cl f_m \, \leq \, \elim_n f_n \, \leq \, \lim_n f_n}, 
\end{equation}
and since $\inf_a (\cl f_m)(a) = \inf_a f_m(a)$, it follows that
\begin{equation} \label{eq-diss-uc-cond-2}
  \textstyle{  \lim_n \inf_a f_n(a) \leq \inf_a \, (\elim_n f_n)(a) \leq \inf_a \left\{  c(x, a) + \int_\X \uh(y) \, q(dy \mid x, a)  \right\}.}
\end{equation}   
For the ACOI to hold, we need equality to hold throughout in (\ref{eq-diss-uc-cond-2}), so if the first inequality in (\ref{eq-diss-uc-cond-2}) is strict, it becomes impossible to obtain the desired ACOI. 
Thus we need the equality 
\begin{equation} \label{eq-diss-uc-cond3}
   \textstyle{\lim_n \inf_a f_n(a) = \inf_a  \, (\elim_n f_n)(a).}
\end{equation}   
Recall that for (UC), we have $ -\infty <  \inf_a  (\elim_n f_n)(a) < + \infty$. In such a case, by \cite[Theorem~7.31]{RoW98}, (\ref{eq-diss-uc-cond3}) holds if and only if for every $\epsilon > 0$, there exists a compact set $K \subset \R^d$ such that 
\begin{equation} \label{eq-diss-uc-cond4}
  \textstyle{ \inf_{a \in K} f_n(a) \, \leq \, \inf_{a} f_n(a) + \epsilon,} \qquad \text{for all $n$ sufficiently large}.
\end{equation}  
In the preceding discussion we have assumed $\A \subset \R^d$  so that we can use the epi-convergence results in \cite[Chap.\ 7]{RoW98} to shorten the discussion. When $\A$ is a general Borel space in the above setup, and also when $\uh_n=h_n$ are not assumed to be nonnegative, one can directly verify that the same conclusion is reached for a compact set $K \subset \A$. 

The inequality (\ref{eq-diss-uc-cond4}) is unwieldy to verify for a given problem, since, among others, $f_n$ depends on the particular choice of the sequence $\{\alpha_n\}$. Therefore, we consider a similar condition instead: For all $\alpha$ sufficiently close to $1$,
$$ \textstyle{ \inf_{a \in K} \left\{ c(x,a) + \alpha \int_\X h_\alpha(y) \, q(dy \mid x, a) \right\} \, \leq \, \inf_{a} \left\{ c(x,a) + \alpha \int_\X h_\alpha(y) \, q(dy \mid x, a) \right\}  + \epsilon.}$$
Since $h_\alpha$ differs from $v_\alpha$ by a constant and with $v_\alpha$ in place of $h_\alpha$, the right-hand side (r.h.s.)  above is $v_\alpha(x) + \epsilon$ by the $\alpha$-DCOE, the above condition is equivalent to Assumption~\ref{cond-uc-2}(i). 

Note that this condition alone does not guarantee the equality (\ref{eq-diss-uc-cond3}) in general when $\uh_n \not= h_n$. Even when (\ref{eq-diss-uc-cond3}) holds, the second inequality in (\ref{eq-diss-uc-cond-2}) can still be strict, ruling out the ACOI. So Assumption~\ref{cond-uc-2}(i) alone is insufficient.
It is Assumption~\ref{cond-uc-2}(ii)-(iii) that give us the rest of the help needed in establishing the ACOI.

We now use a simple example to illustrate when Assumption~\ref{cond-uc-2} holds and can be verified relatively straightforwardly. For the model (UC), because the one-stage cost $c(x,a)$ is bounded over the action set $A(x)$ for a given state $x$, we have not yet found an easy way to identify the compact set $K$ in Assumption~\ref{cond-uc-2}(i) when $A(x)$ is unbounded. (For comparison, in the case of the model (PC) that we will discuss next, an unbounded $c(x,\cdot)$ can help identify the set $K$.) It is for this reason that the example below deals only with a compact action space $\A$.

\begin{example} \label{ex-maj} \rm 
This example is related to \cite[Examples~7.4.2,~8.6.2-8.6.4, and~10.9.3-10.9.5]{HL99}. 
Let $\X = [0, \infty)$, $\A = [0, L]$ for some $L > 0$, 
and with $\ell^+ : = \max \{ \ell, 0\}$, let the states evolve according to 
$$ x_{n+1} = \big[ x_n + \eta_n(x_n, a_n) - \xi_n(x_n,a_n) \big]^+, \qquad n \geq 0. $$
Here we assume that given $(x_n,a_n)=(x,a) \in \X \times \A$, $\eta_n(x, a)$ and $\xi_n(x,a)$ are random variables that are independent of the history 
$\big\{\big(x_k, a_k, \eta_k(x_k, a_k), \xi_k(x_k,a_k) \big)\big\}_{k < n}$ and also mutually independent. Furthermore, we assume that their conditional probability distributions are parametrized by $(x,a)$ only and independent of $n$, and we denote these distributions by $F_{x,a}$ and $G_{x,a}$ for $\eta_n(x, a)$ and $\xi_n(x,a)$, respectively. This specifies indirectly $q(dy \mid x, a)$ for the problem. The admissible action sets $A(x), x \in \X$, are just subsets of $\A$; we do not need them to be compact.

This type of model appears in several applications, e.g., random-release dams, queueing and inventory-production systems (see the aforementioned examples in \cite{HL99}). In a random-release dam model, for example, $x_n$ could be the amount of water in the reservoir at the beginning of the $n$th stage, $\eta_n(x_n,a_n)$ and $\xi_n(x_n,a_n)$ could be the amount of inflow and outflow, respectively, during the $n$th stage. 
In an inventory-production system, $x_n$ could be the stock level and $a_n$ the amount of product ordered at the beginning of the $n$th stage; $\eta_n(x_n,a_n)$ could be the amount of product received and $\xi_n(x_n,a_n)$ the demand during the $n$th stage. 

Define $Z(x,a):  = \eta_0(x, a) - \xi_0(x,a)$. Similarly to \cite[Examples~8.6.4 and 10.9.3]{HL99}, suppose that the collection of probability distributions $F_{x,a}$ and $G_{x,a}$ are such that for all state and admissible action pairs $(x,a) \in \Gamma$, the random variables $Z(x,a)$ have negative means and furthermore, for some $\kappa > 0$,
\begin{equation} \label{eq-ex-lambda}
  \lambda : = \sup_{(x, a) \in \Gamma} \E \left[ e^{\kappa Z(x, a)} \right] < 1.
\end{equation}  
This constraint relates to system stability and can be met, for instance, by choosing the action sets $A(x)$ at each state accordingly.
Then, similarly to the derivations in \cite[Examples~8.6.4, p.~71-73]{HL99}, it can be shown that the weight function $w(x) : = e^{\kappa x}$ satisfies the (UC) model condition (b) (cf.~Def.~\ref{def-uc}) for the constants $b=1$ and $\lambda < 1$ given in (\ref{eq-ex-lambda}):
$$\sup_{a \in A(x)} \textstyle{\int_\X w(y) \, q(dy \mid x, a) } \leq \lambda \, w(x) +1, \qquad \forall \, x \in \X.$$
Suppose that the one-stage cost function also satisfies the bound $\sup_{a \in A(x)} | c(x, a)| \leq \hat c \, e^{\kappa x}$, $x \in \X$, for some $\hat c > 0$, so that the problem under consideration belongs to the (UC) class. Let us discuss now some cases where Assumption~\ref{cond-uc-2} holds. 

Let us choose the compact set $K = \A = [0, L]$ for every state $x$ and $\epsilon > 0$; then Assumption~\ref{cond-uc-2}(i) is satisfied trivially. Next, consider Assumption~\ref{cond-uc-2}(iii) in two different situations:
\begin{enumerate}
\item[(a)] Suppose that for every $(x,a) \in \Gamma$, we have $\eta_0(x,a), \xi_0(x,a) \geq 0$ (which is the case in the aforementioned applications). In addition, suppose that for all $a \in A(x)$, $\eta_0(x,a) \in [0, \ell_x]$ for some state-dependent constant $\ell_x$. Then given $x_0=x$, $x_{1}$ lies in some bounded interval $D_x$. 

Since the exponential weight function $w(y)=e^{\kappa y}$ is continuous, it is bounded above on $D_x$. Then Assumption~\ref{cond-uc-2}(iii) is satisfied trivially, as discussed earlier in Section~\ref{sec-uc-cond}.

\item[(b)] Alternatively, suppose that for each state $x$, $\eta_0(x,a)$ and $\xi_0(x,a)$ are normally distributed random variables that have means and variances bounded uniformly in $a$ over $A(x)$. Or more generally, suppose that they have light-tailed distributions with bounded means and tail probabilities that decrease faster than $e^{- \beta_x |y|}$ as $|y| \to \infty$, for some $\beta_x > 2 \kappa$. 

In this situation, it can be verified straightforwardly that the weight function $w(y) = e^{\kappa y}$ is uniformly integrable w.r.t.\ $\{q(dy \mid x, a) \mid a \in A(x)\}$, as required by Assumption~\ref{cond-uc-2}(iii).
\end{enumerate}

Finally, consider Assumption~\ref{cond-uc-2}(ii) for each state $x$. Suppose w.r.t.\ the Lebesgue measure, $F_{x,a}$, $G_{x,a}$, $a \in A(x)$, have density functions that are bounded above uniformly. Then, in the situation~(a) above, since with $x_0=x$, $x_{1}$ takes values in the bounded set $D_x$, we can simply let the majorizing finite measure $\nu$ in Assumption~\ref{cond-uc-2}(ii) be a multiple of the Lebesgue measure on $D_x$. In the situation~(b) above, there exists a function that not only majorizes all those density functions of $F_{x,a}$, $G_{x,a}$, $a \in A(x)$, but also decays faster than exponentially; subsequently, from this function, one can define a finite measure $\nu$ to satisfy Assumption~\ref{cond-uc-2}(ii). 

Note that we do not need the continuity of $F_{x,a}$ and $G_{x,a}$ in $a$ or in $(x,a)$. Nor do we need the continuity of their density functions.\qed
\end{example}

\subsection{The Case of Nonnegative Costs (PC)} \label{sec-pc-acoi}

We now consider the nonnegative cost model (PC) and introduce two additional assumptions for the study of the ACOI.
The first assumption includes two conditions from the prior work on the ACOI for MDP models satisfying certain continuity conditions.
Let 
$$m_\alpha : = \inf_{x \in \X} v_\alpha(x), \qquad \alpha \in (0,1).$$

\begin{assumption} \label{cond-pc-1}
For the model (PC): 
\begin{enumerate}
\item[\rm (G)] For some policy $\pi$ and state $x$, the average cost $J(\pi,x) < \infty$.
\item[\rm (\underline{B})] For every $x \in \X$, $\liminf_{\alpha \uparrow 1} (v_\alpha(x) - m_\alpha) < \infty$.
\end{enumerate}
\end{assumption}

Precursors to these conditions were introduced in the early work of Sennott \cite{Sen89} and Sch{\"a}l~\cite{Sch93} and then evolved as the research progressed. 
In particular, the condition (G) is the same as that in~\cite{Sch93} and by \cite[Lemma 1.2(b)]{Sch93}, it implies that 
$$ \limsup_{\alpha \to 1} \, (1 - \alpha) \, m_\alpha \leq \inf_{x \in \X} g^*(x) < \infty.$$ 
The condition (\underline{B}) is introduced more recently by Feinberg, Kasyanov, and Zadoianchuk~\cite{FKZ12} to weaken one condition in~\cite{Sch93} (cf.~(\ref{pc-cond-B}) in Example~\ref{ex-2}). That one needs only a pointwise (instead of uniform) upperbound on the relative value functions was first shown by Sennott~\cite{Sen89} for countable-space MDPs. Besides (G) and (\underline{B}), most existing prior work studies the ACOI under additional compactness/continuity conditions; see~\cite{FKZ12,Sch93} for the details of these conditions and their earlier forms (see also \cite{FeL07,VAm15} and the references therein for related work).

Let 
\begin{equation} \label{eq-pc-uh}
    h_\alpha : = v_\alpha - m_\alpha, \qquad  \uh : = \liminf_{\alpha \to 1} h_\alpha.
\end{equation}    
For each $\alpha \in [0, 1)$, define a function $\uh_\alpha$ as 
\begin{equation}
    \uh_\alpha(x) : = \inf_{\beta \in [\alpha, 1)} h_\beta(x), \qquad x \in \X.
\end{equation}    
Note that
\begin{equation}
\uh_\alpha \leq h_\beta \quad \forall \, \alpha \leq \beta < 1 , \qquad \text{and} \qquad \uh_\alpha \uparrow \uh \quad \text{as} \ \alpha \uparrow 1.
\end{equation}

\begin{lem} \label{lem-h-lsa-pc}
{\rm (PC)} Under Assumption~\ref{cond-pc-1}, 
the functions $\uh$ and $\uh_\alpha$, $\alpha \in (0, 1)$, are finite-valued and lower semi-analytic.
\end{lem}

\begin{proof}
That the functions $\uh$ and $\uh_\alpha$ are finite-valued is clear from Assumption~\ref{cond-pc-1}(\underline{B}). To prove that they are lower semi-analytic, we consider them as functions of $(\alpha, x)$, and we show first that $v_\alpha$ is a lower semi-analytic function of $(\alpha, x)$ on $(0, 1) \times \X$. This proof uses the deterministic control model (DM) corresponding to the MDP and is given in Appendix~\ref{appsec-lem} (see Lemma \ref{lem-valpha-lsc}).

Next, consider $m_\alpha$ as a function of $\alpha$.  Since the one-stage cost $c \geq 0$ in (PC), for each policy $\pi$ and initial state $x$, the $\alpha$-discounted value is non-decreasing as $\alpha$ increases. Therefore, $m_\alpha = \inf_{x \in \X} v_\alpha(x)$ is also monotonically non-decreasing as $\alpha$ increases. It follows that $m_\alpha$ is a Borel-measurable function of $\alpha$ on $(0,1)$ and so is $- m_\alpha$.  The latter together with the first part of the proof implies, by \cite[Lemma 7.30(4)]{bs}, that $h_\alpha(x) = v_\alpha(x) - m_\alpha$ is lower semi-analytic in $(\alpha, x)$. 

Now, since for each $\alpha$, $\uh_\alpha$ is the partial minimization of $h_\beta(x)$ over $\beta \in [\alpha, 1)$,  $\uh_\alpha$ is a lower semi-analytic function of $x$ by \cite[Prop.\ 7.47]{bs}. As $\uh$ is the pointwise limit of $\uh_\alpha$ as $\alpha \uparrow 1$, $\uh$ is lower semi-analytic by \cite[Lemma 7.30(2)]{bs}.
\end{proof}

We now impose three additional conditions on the model (PC). They are similar to Assumption~\ref{cond-uc-2} in the previous (UC) case. As before, we use these conditions in place of the compactness/continuity conditions used in the aforementioned prior work, to prove the ACOI for (PC).

\begin{assumption} \label{cond-pc-2}
In the model (PC), for each $x \in \X$ and $\epsilon > 0$, 
\begin{itemize}
\item[\rm (i)-(ii)] Assumption~\ref{cond-uc-2}(i)-(ii) hold;
\item[\rm (iii)] Assumption~\ref{cond-uc-2}(iii) holds with the function $\uh$ in place of the weight function $w$.
\end{itemize}
\end{assumption}

We have already discussed extensively the roles of the conditions (i)-(ii) and how they can be verified, in Sections~\ref{sec-uc-cond},~\ref{sec-uc-acoi-discussion}. An additional observation is that for the (PC) model, verifying the condition (i) can be straightforward when $\A$ is non-compact but $c(x, \cdot)$ is ``coercive'': 

\begin{example} \rm \label{ex-2}
Suppose that for each $x \in \X$, there exists a sequence of compact sets $A_n \subset \A$ such that 
$\inf_{a \not\in A_n} c(x, a) \to + \infty$ as $n \to \infty$. (For instance, $\A = \R^d$ and for each state $x$, $c(x, \cdot)$ is coercive: $c(x, a) \to +\infty$ as $\| a \| \to +\infty$.) Let Assumption~\ref{cond-pc-1}(G) hold and in addition, suppose that a stronger condition than Assumption~\ref{cond-pc-1}(\underline{B}) holds: 
\begin{equation} \label{pc-cond-B}
 \textstyle{\sup_{\alpha \in (0,1)}  \big( v_\alpha(x) - m_\alpha \big) = \sup_{\alpha \in (0,1)} h_\alpha(x) < \infty}, \qquad \forall \, x \in \X.
\end{equation} 
(This is the condition (B) in \cite{Sch93}. Under the condition (G), it is equivalent to $\limsup_{\alpha \uparrow 1} h_\alpha(x) < \infty$ for $x \in \X$~\cite[Lemma 5]{FKZ12}.) Then Assumption~\ref{cond-pc-2}(i) is also satisfied. 
To see this, note that 
$$  \textstyle{\inf_{a \in A(x)} \left\{ c(x, a) + \alpha \int_\X h_\alpha(y) \, q(dy \mid x, a) \right\}  = (1-\alpha) \, m_\alpha + h_\alpha(x),} \qquad \forall \, \alpha < 1,$$
by the $\alpha$-DCOE (Theorem~\ref{thm-dcoe}). The r.h.s.\ is bounded over $\alpha$, in view of (\ref{pc-cond-B}) and an implication of the condition (G), $\sup_{\alpha \in (0,1)} (1-\alpha) m_\alpha < \infty$ \cite[Lemma 1.2(b)]{Sch93}. So, for any $\epsilon > 0$, there exists $n$ sufficiently large such that for any $a \not\in A_n$, $c(x,a) >  (1-\alpha) m_\alpha + h_\alpha(x) + \epsilon$ for all $\alpha \in (0,1)$. Consequently, Assumption~\ref{cond-pc-2}(i) is satisfied by letting $K = A_n$.
\qed
\end{example}

Regarding the condition (iii), it is difficult to verify directly in practice. Instead, a viable way is to first find an upper bound $\hat h \geq \uh$ that has a simple analytical expression like the weight function $w$ in the previous (UC) case, and then verify the inequality (\ref{cond-uc-2c}) with $\hat h$ in place of $\uh$. If such a bound $\hat h$ is available (e.g., obtained in the process of verifying Assumption~\ref{cond-pc-1}(\underline{B})), then the discussions given in Sections~\ref{sec-uc-cond} and~\ref{sec-uc-acoi-discussion} about verifying (\ref{cond-uc-2c}) using the properties of $w$ and the state transition stochastic kernel also apply here with $\hat h$ in place of $w$.

Our ACOI result is as follows. Its conclusions are almost the same as those of Theorem~\ref{thm-uc-acoi} for the (UC) model and the proof is also similar. 

\begin{thm}[the ACOI for (PC)] \label{thm-pc-acoi}
For the (PC) model, suppose Assumptions~\ref{cond-pc-1}-\ref{cond-pc-2} hold. Let $\rho^* = \limsup_{\alpha \uparrow 1} (1 - \alpha) \, m_\alpha$ and let the finite-valued function $\uh \in \An(\X)$ be as given in (\ref{eq-pc-uh}). Then the optimal average cost function $g^*(\cdot) = \rho^*$, and the pair $(\rho^*, \uh)$ satisfies the ACOI:
\begin{equation} \label{eq-pc-acoi}
  \rho^* + \uh(x) \geq \inf_{a \in A(x)} \left\{ c(x, a) + \int_\X \uh(y) \, q(dy \mid x, a) \right\}, \qquad x \in \X.
\end{equation}  
Hence there exist an optimal nonrandomized Markov policy and for each $\epsilon > 0$, an $\epsilon$-optimal nonrandomized stationary policy. 
\end{thm}

\begin{proof}
Consider an arbitrary $x \in \X$ and fix it in the proof below. Choose a sequence $\{\alpha_n\}$ such that 
$$\lim_{n \to \infty} h_{\alpha_n}(x) = \liminf_{\alpha \to 1} h_\alpha(x) = \uh(x).$$ 
Note that $\rho^* \geq \limsup_{n \to \infty} (1 - \alpha_n) \, m_{\alpha_n}$. Recall also that $\rho^* \leq \inf_{y \in \X} g^*(y)$ by~\cite[Lemma 1.2(b)]{Sch93}.
Define $h_n := h_{\alpha_n}$ and $\uh_n := \uh_{\alpha_n}$. Then $\uh_n \leq h_n$ and $\uh_n \uparrow \uh$ as $n \to \infty$.

The rest of the proof is similar to the first part of the proof for Theorem~\ref{thm-uc-acoi}. 
Let $\epsilon > 0$. For the fixed state $x$, by subtracting $\alpha m_\alpha$ from both sides of the $\alpha$-DCOE (Theorem~\ref{thm-dcoe}), we have that for all $n \geq 0$,
\begin{align}
   ( 1 - \alpha_n) \, m_{\alpha_n} + h_n(x) \, & = \, \inf_{a \in A(x)} \left\{ c(x, a) + \alpha_n \int_\X h_n(y) \, q(dy \mid x, a) \right\} \label{eq-prf-thm-pc-oe}.
\end{align}
Then similarly to the derivation of (\ref{eq-prf-thm-uc0a}), we apply Assumption~\ref{cond-pc-2}(i) and replace $h_n$ by $\uh_n$ in the r.h.s.\ above before letting $n \to \infty$ in both sides of (\ref{eq-prf-thm-pc-oe}). This results in the inequality
\begin{align}  \label{eq-prf-thm-pc-oe2}
    \rho^* + \uh(x)     
    & \, \geq \,   \liminf_{n \to \infty} \inf_{a \in K \cap A(x)} \left\{ c(x, a) + \alpha_n \int_\X \uh_n(y) \, q(dy \mid x, a) \right\} - \epsilon
\end{align}   
where $K$ is the compact set in Assumption~\ref{cond-pc-2}(i) for the fixed state $x$ and $\epsilon > 0$, and the integration operation $\int \uh_n(y) \, q(dy \, |\, x, a)$ is valid since $\uh_n$ is universally measurable by Lemma~\ref{lem-h-lsa-pc}. What remains to show is that the conclusion of Lemma~\ref{lem-uc-2acoi} holds so that
\begin{equation} \label{eq-prf-thm-pc-oe3}
\liminf_{n \to \infty} \inf_{a \in K \cap A(x)} \left\{ c(x, a) + \alpha_n \int_\X \uh_n(y) \, q(dy \mid x, a) \right\} 
= \inf_{a \in K \cap A(x)} \left\{ c(x, a) + \int_\X \uh(y) \, q(dy \mid x, a) \right\}.
\end{equation}
For this, we can use the same proof of Lemma~\ref{lem-uc-2acoi} except for two small changes. 
The first change is in the proof of the bound (\ref{eq-prf-thm-uc1c}), $\sup_{a \in K \cap A(x)} \int_{\X \setminus D_\delta} \big(\uh(y) - \uh_n(y) \big)\, q(dy \mid x, a) \leq \eta$ for a given $\eta > 0$ and for all $n$ sufficiently large. Here we use the fact that $\uh - \uh_n \leq \uh$ in this case, and we also use the uniform integrability condition in Assumption~\ref{cond-pc-2}(iii) for the nonnegative function $\uh$, instead of the weight function $w$. The second change is that at the end of the proof of Lemma~\ref{lem-uc-2acoi}, to have $(1 - \alpha_n) \sup_{a \in K \cap A(x)}  \int_\X \uh(y) \, q(dy \mid x, a) \to 0$, we need $\sup_{a \in K \cap A(x)}  \int_\X \uh(y) \, q(dy \mid x, a) < \infty$, and this is implied by the uniform integrability condition on $\uh$ in Assumption~\ref{cond-pc-2}(iii).

We now combine the relations (\ref{eq-prf-thm-pc-oe2}) and (\ref{eq-prf-thm-pc-oe3}) to obtain, as before,
$$  \rho^* + \uh(x) + \epsilon \, \geq \inf_{a \in K \cap A(x)} \left\{ c(x, a) + \int_\X \uh(y) \, q(dy \mid x, a) \right\} \, \geq \, \inf_{a \in A(x)} \left\{ c(x, a) + \int_\X \uh(y) \, q(dy \mid x, a) \right\}.$$
Since this holds for an arbitrary $x \in \X$ and an arbitrary $\epsilon > 0$, we obtain the desired ACOI.
Since $\rho^* \leq g^*(\cdot)$~\cite[Lemma 1.2(b)]{Sch93}, the ACOI just established implies that we must have $g^*(\cdot) = \rho^*$ \cite[Prop.\ 1.3]{Sch93} (the proof is essentially the same as that of Lemma~\ref{lem-optpol} given in Appendix~\ref{appsec-lem}). Finally, the existence of an average-cost optimal Markov policy and $\epsilon$-optimal stationary policy follows from Lemma~\ref{lem-h-lsa-pc}, the ACOI proved above, and Lemma~\ref{lem-optpol}. This completes the proof.
\end{proof}

\section{Other Results: The Existence of a Minimum Pair} \label{sec-4}

In this section, we consider the minimum pair approach (\cite{HLe93,Kur89}; \cite[Chap.\ 5.7]{HL96}) for the case of nonnegative, strictly unbounded costs. Compared with the vanishing discount factor approach, this is a direct approach to analyzing average-cost MDPs.
Let $J(\pi, p)$ denote the average cost of a policy $\pi$ for an initial state distribution $p \in \P(\X)$.
A pair $(\pi^*, p^*) \in \Pi \times \P(\X)$ is called a \emph{minimum pair} if
$$ J(\pi^*, p^*) = \inf_{p \in \P(\X)} \inf_{\pi \in \Pi} J(\pi, p).$$
The question we shall focus on is whether there exists a minimum pair with $\pi^*$ being a stationary policy. When this is the case, under an additional recurrence condition on the Markov chain induced by $\pi^*$, one can ensure that $\pi^*$ is not only optimal but also pathwise optimal for the average cost criterion (see \cite[Theorem 5.7.9(b)]{HL96}), and furthermore, it is also possible to construct from $\pi^*$ a nonrandomized stationary policy with the same property (see \cite[Theorem~2.2 and its proof]{Kur89}).

The minimum pair approach for lower semicontinuous models with strictly unbounded costs has been discussed in detail in the books \cite[Chap.~5.7]{HL96},~\cite[Chap.~11.4]{HL99}. Our interest is again to replace the continuity conditions with majorization type conditions similar to the ones in Section~\ref{sec-3}. In the present case, however, our results are less general: they apply only to a discrete action space. With a discrete $\A$, one can work with Borel measurable policies without encountering measurability issues, so \emph{in this section, we let the policy space $\Pi$ be the set of Borel measurable policies}. The state space $\X$ is still assumed to be a Borel space in our results. (For comparison, for a discrete $\A$, a measure space $\X$ with measurable singleton subsets has been considered~\cite{RiS92}.)

We state the conditions below. The condition (G) is to exclude vacuous problems, and the condition (SU) defines what we mean by strictly unbounded costs. These are standard conditions. The condition (M) is the new majorization condition that we introduce.

\begin{assumption} \label{cond-pc-3}
For the model (PC), the action space $\A$ is a countable space with the discrete topology, the one-stage cost function $c$ and the graph $\Gamma$ of the control constraint are Borel measurable, and furthermore, the following hold:
\begin{enumerate}
\item[\rm (G)] For some policy $\pi$ and state $x \in \X$, the average cost $J(\pi,x) < \infty$.
\item[\rm (SU)] There exist two increasing sequences of compact sets $\{K_n\}$ and $\{A_n\}$ with $K_n \subset \X$ and $A_n \subset \A$, such that
$$ \lim_{n \to \infty} \inf_{(x, a) \not\in K_n \times A_n} c(x, a) = + \infty.$$
\item[\rm (M)] For each set $K \in \{ K_n\}$,  there exist an open set $O \supset K$ and a finite measure $\nu$ on $\B(\X)$ such that
\begin{equation}
     q\big((O \setminus D) \cap B  \mid x, a \big) \leq \nu(B), \qquad \forall \, B \in \B(\X), \ (x, a) \in \Gamma, \label{eq-maj-minpair}
\end{equation}     
where $D \subset \X$ is some closed set (possibly empty) such that restricted to $D \times \A$, the state transition stochastic kernel $q(dy\mid x, a)$ is continuous and the one-stage cost function $c$ is lower semicontinuous. 
\end{enumerate}
\end{assumption}
\smallskip

In the majorization condition (M), roughly speaking, we divide the state space into two parts, a closed set $D$ on which the model has the appealing continuity properties, and the complement set of $D$ on which we impose a majorization condition (\ref{eq-maj-minpair}). The condition (M) is satisfied trivially by letting $D = \X$, if the entire model is lower semicontinuous (this means, in the discrete action setting considered here, that for each $a \in \A$, $q(dy \,|\, x, a)$ is continuous in $x$ and $c(x,a)$ is lower semicontinuous in $x$). For discontinuous models in general, the condition (M) seems natural in cases where the probability measures $\{q(\cdot \mid x, a) \mid (x, a) \in \Gamma\}$ have densities on the complement set $D^c : = \X \setminus D$, w.r.t.\ a common ($\sigma$-finite) reference measure. In such cases, under practical conditions on those density functions, the condition (M) holds; see Section~\ref{sec-minpair-ex} for an illustrative example. 

Our main results are as follows. They are analogous to the prior results for lower semicontinuous models (see \cite[Lemma 2.1, Theorems 2.1, 2.2]{Kur89} and \cite[Lemma 5.7.10 and Theorem 5.7.9(a)]{HL96}).

\begin{prop} \label{prp-su-mp}
{\rm (PC)} Under Assumption~\ref{cond-pc-3}, for any pair $(\pi, p) \in \Pi \times \P(\X)$ with $J(\pi, p) < \infty$, there exists a stationary policy $\bar \mu$ and $\bar p \in \P(\X)$ such that $\bar p$ is an invariant probability measure of the Markov chain on $\X$ induced by $\bar \mu$ and $J(\bar \mu, \bar p) \leq J(\pi, p)$.
\end{prop}

\begin{thm} \label{thm-su-mp}
{\rm (PC)} Under Assumption~\ref{cond-pc-3}, there exists a minimum pair $(\bar \mu, \bar p)$, where $\bar \mu$ is a stationary policy and $\bar p$ is an invariant probability measure of the Markov chain on $\X$ induced by $\bar \mu$.
\end{thm}

\subsection{Proofs}
We prove Prop.~\ref{prp-su-mp} and Theorem~\ref{thm-su-mp} in this subsection. The major proof steps are the same as those for \cite[Theorem 5.7.9 and Lemma 5.7.10]{HL96}, but since the model here is not a lower semicontinuous model, the details in some steps are different. We will focus on those steps. 

Let $\C_b(\X)$ denote the set of bounded continuous functions on $\X$. Recall that a set of probability measures on a topological space is called \emph{tight} if for any $\epsilon > 0$, there exists a compact set whose measure is greater than $1 - \epsilon$ w.r.t.\ every probability measure in that set \cite[p.~293]{Dud02}.

Consider the pair $(\pi, p)$ in Prop.~\ref{prp-su-mp}. Let $\gamma_n \in \P(\X \times \A)$ be the marginal distribution of $(x_n, a_n)$ under the policy $\pi$ with initial distribution $p$. 
For $n \geq 1$, define $\bar \gamma_n \in  \P(\X \times \A)$ to be the average
$$ \textstyle{\bar \gamma_n : = \frac{1}{n} \sum_{k=1}^n \gamma_k.}$$
By assumption, the average cost $J(\pi, p) = \limsup_{n \to \infty} \int c \, d \bar \gamma_n  < \infty$.  
Together with Assumption~\ref{cond-pc-3}(SU), this implies that $\{\bar \gamma_n\}$ is tight. 
So by Prohorov's Theorem~\cite[Theorem 6.1]{Bil68}, there exists a subsequence $\{\bar \gamma_{n_k}\}_{k \geq 0}$ converging weakly to some $\bar \gamma \in \P(\X \times \A)$. 

By \cite[Cor.\ 7.27.2]{bs}, we can decompose $\bar \gamma$ into its marginal $\bar p$ on $\X$ and a Borel measurable stochastic kernel $\bar \mu(da \mid x)$ on $\A$ given $\X$, and by modifying $\bar \mu(da \mid x)$ at a set of $\bar p$-measure zero if necessary, we can make it obey the control constraint of the MDP: 
$$ \bar \gamma (d(x,a)) =  \bar \mu( da \mid x) \, \bar p(dx),  \qquad \text{and}  \qquad \bar \mu(A(x) \mid x) = 1,  \quad \forall \, x \in \X.$$
In order to prove Prop.~\ref{prp-su-mp}, we need to show the following:
\begin{itemize}
\item[(i)] $\bar p$ is an invariant probability measure of the Markov chain on $\X$ induced by the stationary policy $\bar \mu$, namely, 
$$ \bar p(B) = \int_{\X} \int_{\A} q(B \mid x, a) \, \bar \mu(da \mid x) \, \bar p(dx), \qquad \forall \, B \in \B(\X). $$
By \cite[Prop.\ 11.3.2]{Dud02}, this is equivalent to that for every $v \in \C_b(\X)$,
\begin{equation}\label{eq-prf-su2}
   \int_{\X \times \A} \int_{\X} v(y) \, q(dy \mid x, a) \, \bar \gamma (d(x,a)) = \int_{\X} v(x) \, \bar p(dx).
\end{equation}
\item[(ii)] If (i) is established, then $J(\bar \mu, \bar p) = \int c \, d \bar \gamma$ by the invariance property of $\bar p$, so to prove the desired relation $J(\pi, p) \geq J(\bar \mu, \bar p)$, we also need to show that 
\begin{equation} \label{eq-prf-su1}
 \limsup_{n \to \infty} \int c \, d \bar \gamma_n \geq \int c \, d \bar \gamma.
\end{equation}
\end{itemize}
To prove (i)-(ii), we will make use of the following implication of Assumption~\ref{cond-pc-3}(M). 

Let $\bar p_n$ denote the marginal of $\bar \gamma_n$ on $\X$. Recall that $\bar p$ is the marginal of $\bar \gamma$ on $\X$.

\begin{lem} \label{lem-su-mp}
Let the open set $O$, the closed set $D$, and the finite measure $\nu$ on $\B(\X)$ be as in Assumption~\ref{cond-pc-3}(M) for some $K \in \{K_n\}$. Then for all $B \in \B(\X)$,
$$ \bar p \big(  (O \setminus D) \cap B  \big) \leq \nu(B), \qquad  \bar p_n \big( (O \setminus D) \cap B \big) \leq \nu(B), \quad  \forall \, n \geq 1.$$
\end{lem}

\begin{proof}
For $n \geq 1$, consider the marginal distribution $\gamma_n$ of $(x_n, a_n)$. For any $E \in \B(\X)$, $\gamma_n(E \times \A) = \int q(E \mid x, a) \, \gamma_{n-1}(d(x,a))$, so by Assumption~\ref{cond-pc-3}(M), $\gamma_n( \{(O \setminus D) \cap B \}  \times \A) \leq \nu(B)$ for any $B \in \B(\X)$. Since $\bar p_n(\cdot) =  \bar \gamma_n(\cdot \times \A) = \tfrac{1}{n} \sum_{k=1}^n \gamma_k( \cdot \times \A)$, the desired inequality for $\bar p_n$ follows.

We now prove the first inequality for $\bar p$. As the set $O \setminus D$ is open, for any open set $B \subset \X$, the set $(O \setminus D)  \cap B$ is also open. Since $\{\bar \gamma_{n_k} \}$ converges weakly to $\bar \gamma$, by \cite[Theorem 11.1.1]{Dud02}, for any open set $E$, $\bar \gamma (E) \leq \liminf_{k \to \infty} \bar \gamma_{n_k}(E)$. Then, letting $E = \{ (O \setminus D)  \cap B \} \times \A$ for any open set $B$, we have $\bar p((O \setminus D)  \cap B) \leq  \liminf_{k \to \infty} \bar p_{n_k}((O \setminus D)  \cap B) \leq \nu(B)$. This inequality must also hold for any $B \in \B(\X)$. To see this, first, define $\bar p'(\cdot) = \bar p( (O \setminus D)  \cap \cdot)$, to simplify notation. By~\cite[Theorem~7.1.3]{Dud02}, on a metric space, finite Borel measures are closed regular, which means, in our case, that for any Borel set $B$, $\bar p'(B) = \sup\{ \bar p'(F) \mid F \subset B, F \ \text{closed} \}$ and the same is true for $\nu(B)$. This in turn implies that $\bar p'(B) = \inf\{ \bar p'(F) \mid F \supset B, F \ \text{open} \}$ and the same for $\nu(B)$. Now given $B \in \B(\X)$, for any open set $F \supset B$, we have $\bar p'(F) \leq \nu(F)$ as proved earlier, and therefore $\bar p'(B) \leq \nu(B)$.
\end{proof}

We now proceed to prove (\ref{eq-prf-su1}) and then (\ref{eq-prf-su2}).

\begin{lem} \label{lem-prf-su1}
The inequality (\ref{eq-prf-su1}) holds.
\end{lem}

\begin{proof}
For $m \geq 0$, define $c^{m} : \X \times \A \to \R$ by $c^m(x, a) = \min\{ c(x,a), m\}$. 
Since $\liminf_{k \to \infty} \int c \, d \bar \gamma_{n_k} \, \geq \,   \liminf_{k \to \infty} \int c^m \, d \bar \gamma_{n_k}$ and
$\int c^m \, d \bar \gamma \uparrow \int c \, d \bar \gamma$ as $m \to \infty$ by the monotone convergence theorem, 
to prove (\ref{eq-prf-su1}), it suffices to prove
\begin{equation} \label{eq-prf-su3}
 \liminf_{k \to \infty} \int c^m \, d \bar \gamma_{n_k} \geq \int  c^m \, d \bar \gamma.
\end{equation}

To compare the integrals in (\ref{eq-prf-su3}), consider an arbitrary $\epsilon > 0$. There exists a sufficiently large $\bar n$ such that for the corresponding compact sets $K : = K_{\bar n}$ and $F : = A_{\bar n}$ in Assumption~\ref{cond-pc-3}(SU), 
the complement set $(K \times F)^c$ satisfies that
\begin{equation} \label{eq-prf-su4}
 \bar \gamma_{n}\big( (K \times F)^c \big)   \leq \epsilon, \quad \forall \, n \geq 0, \qquad \text{and} \qquad \bar \gamma\big( (K \times F)^c \big)  \leq \epsilon.
\end{equation}
In the above, the existence of such $\bar n$ and the first inequality of (\ref{eq-prf-su4}) follow from Assumption~\ref{cond-pc-3}(SU) and the fact $\limsup_{n \to \infty} \int c \, d \bar \gamma_n < \infty$. The second inequality of (\ref{eq-prf-su4}) follows from the fact that $(K \times F)^c$ is an open set and hence, as the weak limit of $\{\bar \gamma_{n_k}\}$, $\bar \gamma$ satisfies that $\bar \gamma( (K \times F)^c) \leq \liminf_{k \to \infty}  \bar \gamma_{n_k}\big( (K \times F)^c \big)$ by \cite[Theorem 11.1.1]{Dud02}.
We use (\ref{eq-prf-su4}) to bound the integrals $\int_{(K \times F)^c} c^m \, d \bar \gamma_{n_k}$ and $\int_{(K \times F)^c}  c^m \, d \bar \gamma$ by $m \, \epsilon$.

We now compare the integrals of $c^m$ on $K \times F$.
By Assumption~\ref{cond-pc-3}(M), $c$ is lower semicontinuous on the closed set $D \times \A$; therefore, so is $c^m$ on $D \times \A$.
Since the set $F$ is finite, applying Lusin's Theorem~\cite[Theorem 7.5.2]{Dud02}, for any $\delta > 0$, we can choose a closed set $B \subset \X$ with $\nu(\X \setminus B) \leq \delta$ such that $c^m$ is continuous on $B \times F$.
\footnote{Details: We apply Lusin's Theorem for each $a \in F$ to obtain a closed set $B_a \subset \X$ such that $\nu(\X \setminus B_a) \leq \delta/|F|$ and $c^m(\cdot, a)$ is continuous on $B_a$. We then take $B = \cap_{a \in F} B_a$.\label{footnote-lysin}}
Then $c^m$ is lower semicontinuous on the closed set $(D \cup B) \times F$. 
By the Tietze-Urysohn extension theorem~\cite[Theorem 2.6.4]{Dud02} and an approximation property for lower semicontinuous functions \cite[Lemma 7.14]{bs}, the restriction of $c^m$ to $(D \cup B) \times F$ can be extended to a nonnegative lower semicontinuous function $\tilde c^m$ on $\X \times \A$ with the extension also bounded above by $m$.
\footnote{Details: Denote the restriction of $c^m$ to $(D \cup B) \times F$ by $f$. Since it is lower semicontinuous and bounded above by $m$, by \cite[Lemma 7.14]{bs}, there exists a sequence of nonnegative continuous functions $\{f_n\}$ on $(D \cup B) \times F$ with $f_n \uparrow f$. 
We apply the Tietze-Urysohn extension theorem to extend each $f_n$ to a nonnegative continuous function $\tilde f_n$ on $\X \times \A$ that is also bounded above by $m$. We then let $\tilde c^m = \sup_{n} \tilde f_n$.}
For the function $\tilde c^m$, since $\{\bar \gamma_{n_k}\}$ converges weakly to $\bar \gamma$, by \cite[Prop.\ E.2]{HL96},
\begin{equation} \label{eq-prf-su5}
\liminf_{k \to \infty} \int \tilde c^m d \bar \gamma_{n_k} \geq \int \tilde c^m d \bar \gamma.
\end{equation}

We now compare the integrals of $c^m$ with those of $\tilde c^m$ and bound their differences:
\begin{align}
  \left| \, \int_{\X \times \A}  \big(c^m - \tilde c^m\big)  \, d \bar \gamma_{n_k}  \,  \right| & \, = \,  \left| \, \int_{\big((D \cup B) \times F\big)^c}  \big(c^m - \tilde c^m\big) \, d \bar \gamma_{n_k}  \,  \right|  \notag \\
  & \, \leq \, \int_{\big(K \setminus (D \cup B) \big) \times F}  \big| c^m - \tilde c^m \big| \, d \bar \gamma_{n_k}    + \int_{(K \times F)^c}  \big| c^m - \tilde c^m \big| \, d \bar \gamma_{n_k} \notag \\
  & \, \leq \, m \int_{\big(O \setminus (D \cup B) \big) \times \A} \, d \bar \gamma_{n_k} + m \,\epsilon \label{eq-prf-su6a}\\
  & \, \leq \,  m \, \nu(B^c) + m \, \epsilon \label{eq-prf-su6b} \\
  &  \, \leq \, m \, (\delta + \epsilon), \label{eq-prf-su6c}
\end{align}
where the set $O$ in (\ref{eq-prf-su6a}) is the set $O \supset K$ appearing in Assumption~\ref{cond-pc-3}(M), and we used (\ref{eq-prf-su4}) to derive this inequality, and we used Lemma~\ref{lem-su-mp} and the fact $\nu(\X \setminus B) \leq \delta$ to derive (\ref{eq-prf-su6b}) and (\ref{eq-prf-su6c}), respectively. By the same arguments,
\begin{equation}\label{eq-prf-su6d}
 \left| \, \int_{\X \times \A}  \big(c^m - \tilde c^m\big)  \, d \bar \gamma  \,  \right|  \, \leq \,  m \, (\delta + \epsilon).
\end{equation}
By combining (\ref{eq-prf-su6c})-(\ref{eq-prf-su6d}) with (\ref{eq-prf-su5}), we have
$$ \liminf_{k \to \infty} \int c^m \, d \bar \gamma_{n_k} \geq \int c^m \, d \bar \gamma - 2m \, ( \delta + \epsilon).$$
Since $\delta$ and $\epsilon$ are both arbitrary, we obtain
$\liminf_{k \to \infty} \int c^m \, d \bar \gamma_{n_k} \geq \int c^m \, d \bar \gamma$, which is (\ref{eq-prf-su3}) and implies the desired inequality (\ref{eq-prf-su1}) as discussed earlier.
\end{proof}

\begin{lem} \label{lem-prf-su2}
The equality (\ref{eq-prf-su2}) holds.
\end{lem}

\begin{proof}
Recall that $\bar p_{n}$ and $\bar p$ are the marginals of $\bar \gamma_{n}$ and $\bar \gamma$, respectively, on $\X$.
For any $v \in \C_b(\X)$, since $\{\bar \gamma_{n_k}\}$ converges weakly to $\bar \gamma$, it is clear that the r.h.s.\ of (\ref{eq-prf-su2}) satisfies
$$ \int v \, d \bar p = \lim_{k \to \infty} \int v \, d \bar p_{n_k}.$$
The same proof given in \cite[p.\ 119]{HL96} (which is based on a martingale argument) establishes that
$$ \lim_{n \to \infty} \left\{ \int_{\X \times \A} \int_{\X} v(y) \, q(dy \mid x, a) \, \bar \gamma_{n} (d(x,a)) -  \int_{\X} v(x) \, \bar p_{n}(dx)  \right\} \, =  \, 0. $$
Therefore, to prove (\ref{eq-prf-su2}), it suffices to show that for any $v \in \C_b(\X)$,
\begin{equation} \label{eq-prf-su6}
\lim_{k \to \infty} \int_{\X \times \A} \int_{\X} v(y) \, q(dy \mid x, a) \, \bar \gamma_{n_k} (d(x,a)) = \int_{\X \times \A} \int_{\X} v(y) \, q(dy \mid x, a) \, \bar \gamma (d(x,a)).
\end{equation}

Let $\epsilon > 0$ and let the sets $K \subset O$, $F$ and $D$, and the finite measure $\nu$ be as in the proof of Lemma~\ref{lem-prf-su1}. 
Recall that by Assumption~\ref{cond-pc-3}(M), $q(dy \mid x, a)$ is continuous on $D \times \A$.
Since the space $\P(\X)$ is separable and metrizable \cite[Prop.\ 7.20]{bs} and the set $F$ is finite, applying Lusin's Theorem~\cite[Theorem 7.5.2]{Dud02} (cf.~Footnote~\ref{footnote-lysin}), for any $\delta > 0$, we can choose a closed set $B \subset \X$ with $\nu(\X \setminus B) \leq \delta$ such that $q(dy \mid x, a)$ is continuous on $B \times F$.
Then $q(dy \mid x, a)$ is continuous on the closed set $(D \cup B) \times F$, so by \cite[Prop.\ 7.30]{bs}, $\phi(x,a) : = \int_\X v(y) \, q(dy \mid x, a)$ is a bounded continuous function on the closed set $(D \cup B) \times F$, and by the Tietze-Urysohn extension theorem~\cite[Theorem 2.6.4]{Dud02}, this restriction of $\phi$ to $(D \cup B) \times F$ can be extended to a continuous function $\tilde \phi$ on $\X \times \A$ with $\|\tilde \phi\|_\infty \leq \| \phi\|_\infty \leq \| v\|_\infty$.

Since the function $\tilde \phi$ is bounded and continuous and $\{\bar \gamma_{n_k}\}$ converges weakly to $\bar \gamma$, we have
\begin{equation} \label{eq-prf-su7}
\lim_{k \to \infty} \int \tilde \phi  \, d \bar \gamma_{n_k} = \int \tilde \phi  \, d \bar \gamma .
\end{equation}
We now compare the integrals of $\phi$ with those of $\tilde \phi$ and bound their differences, similarly to the derivation of (\ref{eq-prf-su6c})-(\ref{eq-prf-su6d}):
\begin{align*}
 \left| \int_{\X \times \A} \big( \phi - \tilde \phi \big) \, d \bar \gamma_{n_k}  \right| & \, = \,  \left| \int_{ \big( (D \cup B) \times F \big)^c} \big( \phi - \tilde \phi \big) \, d \bar \gamma_{n_k} \right| \\
 & \, \leq \, \int_{\big(K \setminus (D \cup B) \big) \times F}  \big| \phi - \tilde \phi \big| \, d \bar \gamma_{n_k}    + \int_{(K \times F)^c}  \big| \phi - \tilde \phi \big| \, d \bar \gamma_{n_k} \\
 & \, \leq \, 2 \| v\|_\infty \cdot \int_{ \big(O \setminus (D \cup B)\big) \times \A}  d \bar \gamma_{n_k}  + 2 \| v\|_\infty \cdot \epsilon \\
  & \, \leq \, 2 \| v\|_\infty \cdot \nu(B^c) + 2 \| v\|_\infty \cdot \epsilon  \\
  & \, \leq \, 2 \| v\|_\infty \cdot (\delta + \epsilon),
\end{align*}
where we used (\ref{eq-prf-su4}), Lemma~\ref{lem-su-mp}, and the fact $\nu(\X \setminus B) \leq \delta$ by the choice of $B$, to derive the last three inequalities, respectively. By the same arguments, the same conclusion holds for the integrals w.r.t.\ $\bar \gamma$:
$$  \left| \int_{\X \times \A} \big( \phi - \tilde \phi \big) \, d \bar \gamma \, \right| \, \leq \, 2 \| v\|_\infty \cdot (\delta + \epsilon).$$
Combining these two relations with (\ref{eq-prf-su7}), we have
$$ \limsup_{k \to \infty} \left| \int \phi  \, d \bar\gamma_{n_k}  - \int  \phi \, d \bar \gamma  \, \right|  \, \leq \, 4 \| v \|_\infty \cdot (\delta + \epsilon ). $$
Since $\delta$ and $\epsilon$ are arbitrary,  we obtain the desired inequality (\ref{eq-prf-su6}), which implies (\ref{eq-prf-su2}), as discussed earlier.
\end{proof}

\begin{proof}[Proof of Prop.~\ref{prp-su-mp}]
The proposition follows from Lemmas~\ref{lem-prf-su1}-\ref{lem-prf-su2} and the discussion given immediately before Lemma~\ref{lem-su-mp}.
\end{proof}

\begin{proof}[Proof of Theorem~\ref{thm-su-mp}]
By Prop.~\ref{thm-su-mp}, we can construct a sequence of pairs $(\bar \mu_n, \bar p_n)$, $n \geq 1$, such that $\bar \mu_n$ is a stationary policy, $\bar p_n$ an invariant probability measure of the Markov chain on $\X$ induced by $\bar \mu_n$, and $\lim_{n \to \infty} J(\bar \mu_n, \bar p_n) = \inf_{p \in \P(\X)} \inf_{\pi \in \Pi} J(\pi, p) < \infty$. 
Let $\bar \gamma_n(d(x,a)) = \bar \mu_n(da \mid x) \, \bar p_n(dx)$. By the invariance property of $\bar p_n$, we have $J(\bar \mu_n, \bar p_n) = \int c \, d \bar \gamma_n$, so $\int c \, d \bar \gamma_n$ is bounded by some constant for all $n \geq 1$. In view of Assumption~\ref{cond-pc-3}(SU), this implies, as in the preceding proofs, that $\{\bar \gamma_n\}$ is tight and hence there is a subsequence $\{\bar \gamma_{n_k}\}$ converging weakly to some $\bar \gamma \in \P(\X \times \A)$. The rest of the proof now parallels that of Prop.~\ref{prp-su-mp}.
Decompose $\bar \gamma$ into the marginal $\bar p$ on $\X$ and a Borel measurable stochastic kernel $\bar \mu(da \mid x)$ on $\A$ given $X$ that obeys the control constraint. To prove the theorem, we need to show that (\ref{eq-prf-su2}) and (\ref{eq-prf-su1}) hold for $\{\bar \gamma_n\}$ and $\bar \gamma$ in this case.

For all $n \geq 1$, we have $\bar p_n(E)  = \int q(E \mid x, a) \, \bar \gamma_n(d(x,a))$ for all $E \in \B(\X)$, by the invariance property of $\bar p_n$. It follows from this relation and Assumption~\ref{cond-pc-3}(M) that the conclusion of Lemma~\ref{lem-su-mp} holds for $\bar p_n$, and then the second half of the proof of that lemma shows that its conclusion also holds for $\bar p$ in this case. 
We then use Lemma~\ref{lem-su-mp} to prove that for any $v \in \C_b(\X)$, (\ref{eq-prf-su2}) holds. Clearly, the proof amounts to showing that (\ref{eq-prf-su6}) holds, and the arguments are the same as those given in the proof of Lemma~\ref{lem-prf-su2}. This establishes that $\bar p$ is an invariant probability measure of the Markov chain on $\X$ under the stationary policy $\bar \mu$.
Finally, the proof for (\ref{eq-prf-su1}) in this case is exactly the same as the proof of Lemma~\ref{lem-prf-su1}, and this establishes that $\int c \, d\bar \gamma = \lim_{k \to \infty} \int c \, d \bar \gamma_{n_k} = \inf_{p \in \P(\X)} \inf_{\pi \in \Pi} J(\pi, p)$.
Hence $(\bar \mu, \bar p)$ is the desired minimum pair.
\end{proof}

\subsection{An Illustrative Example} \label{sec-minpair-ex}

For simplicity, we consider a problem similar to a one-dimensional linear-quadratic (LQ) control problem but with a discretized action space and piecewise quadratic costs. The same reasoning can be applied to higher dimensional problems with nonlinear dynamics and additive noise. Let $\X = \R$, $\A \subset \R$, and
$c(x,a) = \beta(x) \, ( x^2 +  a^2)$, where $\beta(\cdot)$ is some piecewise constant function such that $\lim\inf_{|x| \to \infty} \beta(x) > 0$, $\sup_{x \in \R} \beta(x) < \infty$.
Let 
$$x_{n+1} = x_{n} + a_n + \zeta_n(x_n, a_n), \quad n \geq 0,$$
where $\zeta_n(x_n, a_n)$ is a random disturbance whose distribution, given the history $\{(x_k, a_k)\}_{k \leq n}$, 
is assumed to depend only on $(x_n, a_n)$ and is given by $F_{x,a} \in \P(\R)$ for $(x_n, a_n)=(x,a)$.
Since our results only apply to discrete action spaces, as an illustrative example, consider a small number $\delta > 0$ and let 
$\A = \big\{ k \delta \mid k = 0, \, \pm 1, \, \pm 2, \, \ldots \big\}.$
The one-stage costs in this problem are strictly unbounded, and we can let the compact subsets of states and actions in Assumption~\ref{cond-pc-3}(SU) be $K_n = [-n, n]$ and $A_n = \big\{ 0, \, \pm \delta, \,  \pm 2 \delta, \, \ldots, \, \pm n \delta \big\}$  for $n \geq 1$, for example.

Now similar to what we discussed in Section~\ref{sec-uc-cond}, suppose that w.r.t.\ the Lebesgue measure, the distributions $F_{x,a}$ have densities $f_{x,a}$ that are bounded uniformly from above by $\ell$.
For a closed interval $K = [-n, n]$, consider the open interval $O = (-n-1, n+1)\supset K$. 
A finite measure $\nu$ satisfying Assumption~\ref{cond-pc-3}(M) is simply given by $\ell$ times the Lebesgue measure on the open interval~$O$. (We let $D = \emptyset$ in Assumption~\ref{cond-pc-3}(M) in this case.)

Regarding Assumption~\ref{cond-pc-3}(G), suppose that $F_{x,a}$, $(x, a) \in \Gamma$, have zero means and variances bounded uniformly by $\sigma^2$. Let $A(x) = \A$ for all $x \in \X$, for simplicity. Then a policy $\pi$ that satisfies Assumption~\ref{cond-pc-3}(G) for the initial state $x_0=0$ is the one that chooses the action $a_n = \argmin_{a \in \A, |a| \leq |x_n|} | x_n + a|$, since $J(\pi, 0) \leq 2  (\delta^2 + \sigma^2) \cdot \sup_{x \in \R} \beta(x) < \infty$, as can be verified. In this simple example, we can see immediately a solution to the condition (G). For more complicated problems, one can use Markov chain theory in finding a stationary policy $\pi$ with a finite average cost; see \cite[Sec.~IV.A]{Mey97}.

Thus, Theorem~\ref{thm-su-mp} holds in this example. Like in the previous Section~\ref{sec-3}, there is no need here for the continuity of $q(dy \mid x, a)$ or $c(x,a)$ in $x$.

\section*{Acknowledgment}
The author would like to thank Professor Eugene Feinberg, who gave helpful comments on a preliminary draft of this manuscript and pointed her to several related prior results on Borel-space MDPs, and Dr.~Martha Steenstrup, whose suggestions helped her improve the presentation.
This research was supported by a grant from Alberta Innovates---Technology Futures.


\markboth{On Average-Cost Borel-Space MDPs}{}

\begin{appendices}

\section{Some Basic Optimality Properties for the Average and Discounted Cost Criteria} \label{appsec-opt}

In this appendix, we first define a deterministic control model (DM) that corresponds to a given Borel-space MDP, which will be referred to below as (SM) for short (this is the term used in \cite{bs}, standing for the stochastic control model). This background material, given in Section~\ref{appsec-dm}, is based on the book~\cite[Chap.\ 9.2]{bs}, and it is an important part of the theory for Borel-space MDPs with universally measurable policies, as we explained in Section~\ref{sec-2.2}. The text of Section~\ref{appsec-dm} is similar to an online appendix the author wrote earlier for a paper on total cost Borel-space MDPs~\cite{yu-tc15}. 

The rest of this appendix (Sections~\ref{sec-prf-thmacbasic}-\ref{appsec-lem}) contains the proof details for the results given in Section~\ref{sec-2.2}, as well as a technical lemma used in proving Lemma~\ref{lem-h-lsa-pc} in Section~\ref{sec-pc-acoi}. 
These proofs make use of the correspondence relations between (DM) and (SM) explained in Section~\ref{appsec-dm}, in addition to several theorems from \cite[Part II]{bs}. That is why we have collected them together in this appendix.

\subsection{Correspondence between a Borel-space MDP and a Deterministic Control Model} \label{appsec-dm}

Recall that $\P(\X)$ and $\P(\X \times \A)$ denote the sets of Borel probability measures on the state space $\X$ and the state-action space $\X \times \A$, respectively. The set $\Gamma = \{ (x,a) \mid x \in \X, a \in A(x) \}$ is the graph of the control constraint $A(\cdot)$ of the Borel-space MDP defined in Section~\ref{sec-2.1.2}, and it is assumed to be an analytic set. 

Recall also that, as explained in Section~\ref{sec-2.1}, the integral $\int f dp$ for a universally measurable function $f$ and a Borel probability measure $p$ is defined as the integration of $f$ w.r.t.\ the completion of $p$. For summations involving extended real numbers, the convention $-\infty + \infty = \infty - \infty = \infty$ is used (cf.~\cite[Chap.\ 7.4.4]{bs}). For the model (UC) we consider, such summations can appear in the analysis of the corresponding (DM) but will not appear in that of the original problem (SM).

\begin{definition}[cf.\ {\cite[Defs.\ 9.4-9.6 and 9.7]{bs}}]  \label{def-DM} 
For a Borel-space MDP defined in Section~\ref{sec-2.1.2}, 
the \emph{corresponding deterministic control model (DM)} is defined as follows. 
\begin{enumerate}
\item[(i)] The state and action spaces and the model parameters are given by:
\begin{itemize}
\item[$\bullet$] State space $\P(\X)$ and action space $\P(\X \times \A)$.
\item[$\bullet$] Control constraint $\bar A(\cdot)$, which maps each state $p \in \P(\X)$ to a set $\bar A(p)$ of admissible actions at $p$, defined as
$$ \bar A(p) : = \big\{ \q \in \P(\X \times \A) \ \big|  \ \q(B \times \A) = p(B), \ \q(\Gamma) = 1, \, \forall \, B \in \B(\X) \big\}.$$
I.e., the actions at state $p$ are those Borel probability measures on $\X \times \A$ that have $p$ as the marginal on $\X$ and assign probability one to the graph $\Gamma$ of the original control constraint $A(\cdot)$. (Like $A(\cdot)$, the graph of $\bar A(\cdot)$ is an analytic set in the state-action space $\P(\X) \times \P(\X \times \A)$; see the proof of~\cite[Lemma 9.1]{bs}.)
\item[$\bullet$] System function $\bar f: \P(\X \times \A) \to \P(\X)$, which maps each action $\q \in \P(\X \times \A)$ to a state $\bar f(\q) \in \P(\X)$, defined according to the state transition stochastic kernel of the original problem as
$$ \bar f(\q)(B) : = \textstyle{\int_{\X \times \A} q(B \!\mid x, a) \, \q\big(d(x,a)\big)}, \qquad \forall \, B \in \B(\X).$$
This function $\bar f$ specifies how the states evolve in (DM). 
\item[$\bullet$] One-stage cost function $\bar c : \P(\X \times \A) \to [-\infty, + \infty]$, defined by the one-stage cost function of the original problem as
$$ \bar c(\q) : = \textstyle{\int_{\X \times \A} c(x,a) \, \q \big(d(x,a)\big).} $$
(Like $c(\cdot)$, the function $\bar c(\cdot)$ is lower semi-analytic~\cite[Prop.\ 7.48]{bs}.)
\end{itemize}
\item[(ii)] A \emph{policy for (DM)} is a sequence of mappings $\bar \pi = (\bar \mu_0, \bar \mu_1, \ldots)$ such that for each $k \geq 0$, $\bar \mu_k : \P(\X) \to \P(\X \times \A)$ and $\bar \mu_k(p) \in \bar A(p)$ for every $p \in \P(\X)$. In other words, $\bar \mu_k$ maps each state $p$ to an admissible action at $p$. (Note that every policy here is Markov and there are no measurability conditions on them in this deterministic control problem.) The set of all policies for (DM) is denoted by $\bar \Pi$.
\item[(iii)] Given an initial state $p_0 \in \P(\X)$, applying a policy $\bar \pi$ in (DM) generates recursively a sequence of state and action pairs $(p_0, \q_0), (p_1, \q_1), \ldots$, according to the system function $\bar f$ as
\begin{equation}
  \q_k = \bar \mu_k(p_k), \qquad p_{k+1}=\bar f(\q_k), \quad k \geq 0.
\end{equation}  
Such a sequence $(p_0, \q_0, \q_1, \ldots)$ is called an \emph{admissible sequence}. The set of all admissible sequences that can be generated by the policies in (DM) is denoted by $\Delta$ (which is an analytic set~\cite[Lemma 9.1]{bs}).
\qed
\end{enumerate}
\end{definition} 
\medskip

The policies of the original problem (SM) and (DM) are related (cf.\ \cite[Def.\ 9.9 and Prop.\ 9.2]{bs}):
\begin{itemize}
\item[(a)] Given $p_0 \in \P(\X)$ and a policy $\pi \in \Pi$ of (SM), we can define a policy $\bar \pi$ of (DM) that corresponds to $\pi$ at $p_0$ in the following sense. Let $p_k \in \P(\X)$ and $\q_k \in \P(\X \times \A)$, $k \geq 0$, be the marginal distributions of the state $x_k$ and the state-action pairs $(x_k,a_k)$ in (SM), respectively, at time $k$, under the policy $\pi$ and with the initial state distribution being $p_0$. Then we can define a policy $\bar \pi = (\bar \mu_0, \bar \mu_1, \ldots)$ for (DM) such that those marginal distributions $\q_k, p_{k+1}$, $k \geq 0$, are exactly the actions and states generated by the policy $\bar \pi$ in (DM) according to Def.~\ref{def-DM}(iii), when the initial state is $p_0$. Furthermore, if $\pi$ is Markov, there exists a policy $\bar \pi$ of (DM) that corresponds to $\pi$ at every $p_0 \in \P(\X)$ in the above sense.
\item[(b)] Conversely, given $p_0 \in \P(\X)$ and a policy $\bar \pi$ of (DM), we can define a Markov policy $\pi$ of (SM) such that the two policies correspond at $p_0$ in the sense explained above.
\footnote{The policy $\pi$ is constructed by decomposing the action sequence $\{\q_k\}$ generated from the initial state $p_0$ by the policy $\bar \pi$ in (DM) according to Def.~\ref{def-DM}(iii). Specifically, each $\q_k$ is decomposed into the marginal $p_k$ on $\X$ and a stochastic kernel $\mu_k$ on $\A$ given $X$ that obeys the original control constraint. Then $\pi$ is defined to be the collection of the stochastic kernels $(\mu_0, \mu_1, \ldots)$. See \cite[Prop.\ 9.2]{bs} and its proof for further details.}
\end{itemize}

Let us now consider the average cost problem in (DM) and relate it to the average-cost MDP.
In (DM), for an initial state $p_0 \in \P(\X)$, we denote the $n$-stage cost and the average cost of a policy $\bar \pi$ by $\bar J_n(\bar \pi, p_0)$ and $\bar J(\bar \pi,p_0)$, respectively. Since the problem is deterministic, the former is given by
$\bar J_n(\bar \pi, p_0) = \frac{1}{n} \sum_{k=0}^{n-1} \bar c(\q_k)$
(recall the convention $-\infty + \infty = \infty-\infty=\infty$), where $\{\q_k\}$ is the action sequence generated by $\bar \pi$ according to Def.~\ref{def-DM}(iii). 
As in (SM), the {optimal average cost} at $p_0$ is defined as
$$ \bar g^*(p_0) : = \inf_{\bar \pi \in \bar \Pi} \bar J(\bar \pi,p_0) = \inf_{\bar \pi \in \bar \Pi}  \limsup_{n \to \infty}  \bar J_n(\bar \pi, p_0)/ n.$$

Consider the value $\bar g^*(\delta_x)$ for the Dirac measure $\delta_x$ at $x \in \X$.
Note first that by the correspondence between (DM) and (SM) described in (b) above, if (SM) is in the model class (UC), then in its corresponding (DM), $\bar J_n(\bar \pi, \delta_x)$ is finite for all $x \in \X$, $\bar \pi \in \bar \Pi$ and $n \geq 0$, whereas if (SM) is in the model class (PC), $\bar J_n(\bar \pi, \delta_x)$ is nonnegative and possibly infinite. So for both (PC) and (UC),
\begin{equation} \label{eq-app-bg-finite}
  \bar g^*(\delta_x) > - \infty, \qquad \forall \, x \in \X.
\end{equation}  
More importantly, by the relations described in (a)-(b) above, the average costs of the corresponding policies in (DM) and (SM) are equal as well:
\begin{itemize}
\item[($\text{a}'$)] Given $x \in \X$ and a policy $\pi$ of (SM), there exists a policy $\bar \pi$ of (DM) such that $J(\pi,x) = \bar J(\bar \pi,\delta_x)$. 
\item[($\text{b}'$)] Conversely, given $x \in \X$ and a policy $\bar \pi$ of (DM), there exists a Markov policy $\pi$ of (SM) such that $J(\pi,x) = \bar J(\bar \pi,\delta_x).$ 
\end{itemize}
Thus we must have $g^*(x) = \bar g^*(\delta_x)$ for all $x \in \X$.

\subsection{Proof of Theorem~\ref{thm-ac-basic}} \label{sec-prf-thmacbasic}

We prove Theorem~\ref{thm-ac-basic} in this subsection. We first prove several average cost optimality properties for (DM), and we then transfer the results to (SM) via the correspondence relations between the two models given in the preceding subsection.

\medskip
\noindent {\bf Deriving several desired optimality properties for (DM):} \ 
For (DM), we can write the optimal average cost function $\bar g^*$ as the result of partial minimization of a lower semi-analytic function as follows. Recall the set $\Delta$, which consists of all admissible sequences $(p_0, \q_0, \q_1, \ldots)$ (cf.~\cite[Def.\ 9.7]{bs} and Def.~\ref{def-DM}(iii)). For each $p_0 \in \P(\X)$, denote the vertical section of $\Delta$ at $p_0$ by $\Delta_{p_0}$, i.e.,
$$ \Delta_{p_0} :  = \big\{(\q_0, \q_1, \ldots) \mid (p_0, \q_0, \q_1, \ldots)  \in \Delta \big\};$$
$\Delta_{p_0}$ is the set of all action sequences $(\q_0, \q_1, \ldots)$ that can be generated by some policy of (DM) for the initial state $p_0$. Define a function $G: \Delta \to [-\infty, + \infty]$ by 
$$ G \big(p_0, \q_0, \q_1, \ldots \big) : = \limsup_{n \to \infty} \textstyle{\frac{1}{n} \sum_{k=0}^{n-1} \bar c(\q_k)}.$$
Then $\bar g^*$ can be written equivalently as the result of partial minimization of $G$:
\begin{equation} \label{eq-bJ}
    \bar g^*(p_0) = \inf_{(\q_0, \q_1, \ldots) \in \Delta_{p_0}} G \big(p_0, \q_0, \q_1, \ldots \big).
\end{equation}

A crucial result is that the set $\Delta$ is an analytic subset of the space $\P(\X) \times \big(\P(\X \times \A) \big)^\infty$ (endowed with the product topology) \cite[Lemma 9.1]{bs}. 
Then, since the sum of a finite number of lower semi-analytic functions are lower semi-analytic \cite[Lemma 7.30(4)]{bs} and $\limsup_{n \to \infty} f_n$ is lower semi-analytic for a sequence $\{f_n\}$ of lower semi-analytic functions \cite[Lemma 7.30(2)]{bs}, the function $G$ is lower semi-analytic. Consequently, being the result of partial minimization of $G$ (cf.\ (\ref{eq-bJ})),  the function $\bar g^*$ is lower semi-analytic by \cite[Prop.\ 7.47]{bs}.  

Moreover, by a measurable selection theorem \cite[Prop.\ 7.50]{bs}, for any $\epsilon > 0$, we can select a measurable $\epsilon$-minimizer for the optimization problem in (\ref{eq-bJ}). More precisely, there exists a universally measurable mapping $\psi: \P(\X) \to \big(\P(\X \times \A)\big)^\infty$ such that
for all $p_0 \in \P(\X)$,
$$  \psi(p_0) \in \Delta_{p_0} \ \ \ \text{and} \ \ \ G\big(p_0, \psi(p_0) \big) \leq \begin{cases}  
    \bar g^*(p_0)+\epsilon, & \text{if} \  \bar g^*(p_0) > - \infty; \\
    - 1/\epsilon, & \text{if} \ \bar g^*(p_0) = - \infty.  \end{cases} $$
Furthermore, by \cite[Prop.\ 7.50(b)]{bs}, $\psi$ can be chosen to attain the infimum in (\ref{eq-bJ}) for all $p_0$ at which this is possible; i.e.,
\begin{equation} \label{eq-bJ-min2}
   G\big(p_0, \psi(p_0) \big) = \bar g^*(p_0), \qquad \forall \, p_0 \in \I, 
\end{equation}
where 
\begin{equation} 
 \I = \Big\{ \, p_0 \in \P(\X) \, \, \Big| \,  \, \exists \, (\q_0, \q_1, \ldots) \in \Delta_{p_0} \text{ with } \bar g^*(p_0) = G \big(p_0, \q_0, \q_1, \ldots \big)  \,  \Big\}. \notag
\end{equation}
As noted earlier in (\ref{eq-app-bg-finite}), $\bar g^*(\delta_x) > - \infty$ for all $x \in \X$ if the original problem (SM) is in the class (PC) or (UC). Therefore,
\begin{equation} \label{eq-bJ-min}
  \psi(\delta_x) \in \Delta_{\delta_x} \ \ \ \text{and} \ \ \ G\big(\delta_x, \psi(\delta_x) \big) \leq \bar g^*(\delta_x)+\epsilon, \qquad \forall\, x \in \X.
\end{equation}  
By (\ref{eq-bJ-min2}), in the case $\delta_x \in \I$ for all $x \in \X$, $\psi$ can be chosen to attain the minimal values at those Dirac measures:
\begin{equation} \label{eq-bJ-min1}
   \psi(\delta_x) \in \Delta_{\delta_x} \ \ \ \text{and} \ \ \ G\big(\delta_x, \psi(\delta_x) \big) = \bar g^*(\delta_x), \qquad \forall\, x \in \X.
\end{equation}

We now transfer these results to (SM). 

\medskip
\smallskip
\noindent {\bf Proof of the first part of the theorem:} \ The argument is the same as that in \cite[Chap.\ 9.3]{bs}. Since $\bar g^*$ is lower semi-analytic and the mapping $x \mapsto \delta_x$ is a homeomorphism \cite[Cor.\ 7.21.1]{bs}, the function $g^*(x) = \bar g^*(\delta_x)$ is lower semi-analytic by \cite[Lemma 7.30(3)]{bs}. This proves Theorem~\ref{thm-ac-basic}(i).

\medskip
\smallskip
\noindent {\bf Proof of the second part of the theorem:} \ The second part of the theorem concerns the existence of universally measurable, $\epsilon$-optimal or optimal semi-Markov policies in (SM), 
and to prove it, we start from (\ref{eq-bJ-min}) or (\ref{eq-bJ-min1}) and construct a policy for (SM) with desired properties. We prove the existence of an $\epsilon$-optimal policy first; the proof for the existence of an optimal policy is largely the same and will be given at the end.

For the average cost problem (\ref{eq-bJ}) in (DM), (\ref{eq-bJ-min}) gives an $\epsilon$-optimal solution $\sigma(x) = \psi(\delta_x)$ for $x \in \X$, which is universally measurable in $x$ by the universal measurability of $\psi(\cdot)$ together with \cite[Cor.~7.21.1 and Lemma 7.30(3)]{bs}. 
Express this solution as 
$$\sigma(x) = (\q_0(x), \q_1(x), \ldots)$$ 
where each $\q_k: \X \to \P(\X \times \A)$ is a universally measurable stochastic kernel on $\X \times \A$ given $\X$.
We then apply \cite[Prop.\ 7.27]{bs} to decompose each $\q_k$, $k \geq 0$, into two universally measurable stochastic kernels $\mu_k(d a_k \mid x, x_k)$ and $p_k(dx_k \mid x)$ that satisfy
\begin{equation} \label{eq-app-prf-acpol0}
  \q_k(x)(B) = \int_\X \int_\A \ind_B(x_k, a_k) \, \mu_k(d a_k \mid x, x_k) \, p_k(dx_k \mid x), \qquad \forall \, B \in \B(\X \times \A),
\end{equation}  
where $\ind_B(\cdot)$ denotes the indicator function for the set $B$.
For $k = 0$, by the control constraint in (DM) (cf.~Def.~\ref{def-DM}(i)), we have $p_0(dx_0 \mid x) = \delta_x$ and $\q_0(x)(\Gamma) = 1$, so $\mu_0(d a_0 \mid x, x)$ must satisfy the control constraint of (SM):
\begin{equation} \label{eq-app-prf-acpol1a}
 \mu_0\big( A(x) \mid x, x \big) = 1, \qquad \forall \, x \in \X.
\end{equation}
For $k \geq 1$, the stochastic kernels $\mu_k$ need not satisfy the control constraint of (SM). However, for each $x$, since $\q_k(x)(\Gamma) = 1$ by the control constraint in (DM) (cf. Def.~\ref{def-DM}(i)), we have, by (\ref{eq-app-prf-acpol0}),
$$ \int_\X \mu_k( A(x_k) \mid x, x_k) \, p_k(dx_k \mid x) = 1,$$
which implies that for each $x$,
\begin{equation} \label{eq-app-prf-acpol1}
 \mu_k( A(x_k) \mid x, x_k) = 1,  \qquad \text{$p_k(dx_k \mid x)$-almost surely}. 
\end{equation}
We now alter each $\mu_k$ at those points $(x, x_k)$ where the control constraint in (SM) is violated, in order to make the collection of these stochastic kernels a valid policy for (SM).

To this end, for each $k \geq 1$, define a set 
\begin{equation} \label{eq-app-prf-acpol2}
 D_k : = \big\{ (x, x_k) \mid \mu_k( A(x_k) \mid x, x_k) < 1 \big\}.
\end{equation} 
Since 
$$ \mu_k\big( A(x_k) \mid x, x_k \big) =  \int_\A \ind_\Gamma(x_k, a_k) \, \mu_k(da_k \mid x, x_k) $$
and the indicator function $\ind_\Gamma(x_k, a_k)$ is universally measurable (because $\Gamma$ is analytic),
by \cite[Prop.~7.46]{bs}, the function $\mu_k( A(x_k) \mid x, x_k)$ is universally measurable in $(x, x_k)$.
Consequently, the set $D_k$ is universally measurable.
Then for each $x$, as the vertical section of $D_k$, $D_{k,x} : = \{ x_k \mid (x, x_k) \in D_k\}$ is also universally measurable \cite[Lemma 7.29]{bs}, and by (\ref{eq-app-prf-acpol1}),
\begin{equation} \label{eq-app-prf-acpol1b}
    p_k\big( D_{k,x} \mid x \big)  = 0, \qquad \forall x \in \X.
\end{equation} 
Recall that since $\Gamma$ is analytic, by the Jankov-von Neumann selection theorem \cite[Prop.\ 7.49]{bs}, there exists a universally (in fact, analytically) measurable stationary policy $\mu^o$.
We now define stochastic kernels $\tilde \mu_k$, $k \geq 1$, by altering each $\mu_k$ as follows:
\begin{align} \label{eq-app-prf-acpol3}
    \tilde \mu_k(da_k \mid x, x_k) = \begin{cases}
          \mu_k(da_k \mid x, x_k), \quad &  \text{if} \ (x, x_k) \not\in D_k;\\
          \mu^o(da_k \mid x_k), \quad & \text{if} \ (x, x_k) \in D_k.
      \end{cases}
\end{align}
By definition each $\tilde \mu_k$ above is a universally measurable stochastic kernel that obeys the control constraint in (SM).
For $k = 0$, let
$$ \tilde \mu_0(da_0 \mid x) = \mu_0(da_0 \mid x, x).$$
By \cite[Prop.\ 7.44]{bs}, $\tilde \mu_0$ is universally measurable, and by (\ref{eq-app-prf-acpol1a}), it satisfies the control constraint in (SM).
Then, identifying $x$ with $x_0$, we have that
$$\pi = \big(\tilde \mu_0(da_0 \mid x), \, \tilde \mu_1(da_1 \mid x, x_1), \, \tilde \mu_2(da_2 \mid x, x_2), \, \ldots \big)$$ 
is a universally measurable, semi-Markov policy for (SM).

What remains to prove is that $\pi$ is an $\epsilon$-optimal policy. This follows from (\ref{eq-app-prf-acpol1a})-(\ref{eq-app-prf-acpol1}), (\ref{eq-app-prf-acpol1b})-(\ref{eq-app-prf-acpol3}) and the correspondence relations between $\pi$ and a policy in (DM) that generates the admissible sequence 
$\big(\delta_x, \sigma(x)\big) = \big(\delta_x, \q_0(x), \q_1(x), \ldots \big)$ for each $x \in \X$. 
Specifically, in (SM), under $\pi$, if the initial state $x_0 = x$,  the marginal distribution $\tilde \q_k(x)$ of $(x_k, a_k)$ for $k \geq 1$ is given by
$$ \tilde \q_k(x)(B) : = \int_\X \int_\A \ind_B(x_k, a_k) \, \tilde \mu_k(da_k \mid x, x_k)  \, \tilde p_k( d x_k \mid x), \qquad \forall \, B \in \B(\X \times \A), $$
where $\tilde p_k (d x_k \mid x)$ is the marginal distribution of $x_k$, and for $k = 0$,  the marginal distributions of $x_0$ and $(x_0, a_0)$ satisfy
$$ \tilde p_0 (d x_0 \mid x) = \delta_x = p_0 (d x_0 \mid x), \qquad \tilde \q_0(x) = \q_0(x),$$
in view of the definitions of $\tilde \mu_0$ and $\mu_0$.
We now use induction on $k$ and the relation~(\ref{eq-app-prf-acpol1}) to verify that $\tilde p_k (d x_k \mid x) = p_k (d x_k \mid x)$ and $\tilde \q_k(x) = \q_k(x)$ for all $k \geq 0$. Suppose that for some $k \geq 1$, we have shown that they hold for $0, 1, \ldots k-1$; let us prove that they also hold for $k$. Since $\tilde \q_{k-1}(x) = \q_{k-1}(x)$ by assumption, the marginal distributions of $x_k$ are also equal: $\tilde p_k(dx_k \mid x) = p_k(dx_k \mid x)$, in view of the definition of the transition function in (DM) (cf.~Def.~\ref{def-DM}(i)).
Together with (\ref{eq-app-prf-acpol1b}) and the definition (\ref{eq-app-prf-acpol3}) of $\tilde \mu_k$, this implies that for any $B \in \B(\X \times \A)$,
\begin{align*}
   \int_\X \int_\A \ind_{B}(x_k, a_k) \, \tilde \mu_k(da_k \mid x, x_k) \, \tilde p_k(dx_k \mid x)  & =  \int_\X \int_\A \ind_{B}(x_k, a_k) \, \tilde \mu_k(da_k \mid x, x_k) \, p_k(dx_k \mid x) \\
   &  =\int_{\X \setminus D_{k,x}} \int_\A \ind_{B}(x_k, a_k) \, \tilde \mu_k(da_k \mid x, x_k) \, p_k(dx_k \mid x) \\
   & = \int_{\X \setminus D_{k,x}} \int_\A \ind_{B}(x_k, a_k) \, \mu_k(da_k \mid x, x_k) \, p_k(dx_k \mid x) \\
   & = \int_{\X} \int_\A \ind_{B}(x_k, a_k) \, \mu_k(da_k \mid x, x_k) \, p_k(dx_k \mid x).
\end{align*}   
This shows $\tilde \q_k(x) = \q_k(x)$. So by the induction argument, we have $\tilde \q_k(x) = \q_k(x)$ for all $k \geq 0$.

The result just proved implies that for each $x \in \X$, the average cost of the policy $\pi$ equals $G\big(\delta_x, \psi(\delta_x) \big) \leq \bar g^*(\delta_x) + \epsilon$ (cf.\ (\ref{eq-bJ-min})). As proved earlier, $g^*(x) = \bar g^*(\delta_x)$ for $x \in \X$, so $\pi$ is $\epsilon$-optimal. This proves the existence of a universally measurable, $\epsilon$-optimal, randomized semi-Markov policy.

Finally, consider the last statement of the theorem. By assumption, for each $x \in \X$, there exists a policy in (SM) attaining the optimal average cost $g^*(x)$. By the correspondence relations between (SM) and (DM) discussed at the end of Section~\ref{appsec-dm}, this means that for each $x \in \X$, there exists a policy in (DM) attaining the optimal average cost $\bar g^*(\delta_x)$. 
Thus we can choose a solution $\psi(\cdot)$ of the average cost problem (\ref{eq-bJ}) in (DM) to satisfy (\ref{eq-bJ-min1}).
Then, letting $\sigma(x) = \psi(\delta_x)$ for $x \in \X$ in the preceding proof, we obtain, with exactly the same arguments, that the semi-Markov policy $\pi$ constructed above for (SM) has average costs $J(\pi, x) = G\big(\delta_x, \psi(\delta_x) \big) = \bar g^*(\delta_x) = g^*(x)$ for all $x \in \X$. 
So $\pi$ is a universally measurable, optimal, randomized semi-Markov policy.
This completes the proof of Theorem~\ref{thm-ac-basic}(ii).

\subsection{Proof of Theorem~\ref{thm-dcoe} for (UC)} \label{appsec-prf-uc}

Note first that the two conditions (a)-(b) in Definition~\ref{def-uc} of the model (UC) ensure that for all policies $\pi$, starting from $x \in \X$, the expected $\alpha$-discounted total cost $v^\pi_\alpha(x)$ is bounded in absolute value by 
$$ \hat c \cdot \sum_{n = 0}^\infty \alpha^n \,\E^\pi_x \big[ w(x_n) \big]  \leq \frac{\hat c \, w(x)}{1 - \alpha \lambda} + \frac{\hat c \, b}{(1-\alpha)(1 - \lambda)}.$$
Therefore, the optimal value function $v_\alpha$ is not only finite everywhere but also satisfies $\|v_\alpha\|_w < \infty$.

We now consider the corresponding deterministic control model (DM) described in Section~\ref{appsec-dm}. 
For a policy $\bar \pi \in \bar \Pi$ and initial state $p_0 \in \P(\X)$, define its $\alpha$-discounted total cost to be
$$ \bar v_{\alpha}^{\bar \pi}(p_0) : = \limsup_{n \to \infty} \textstyle{\sum_{k=0}^{n-1} \alpha^k  \, \bar c(\q_k)}$$
where $\{\q_k\}$ is the action sequence generated by $\bar \pi$ according to Def.~\ref{def-DM}(iii).
Define the optimal $\alpha$-discounted value function in (DM) as 
$$\bar v_{\alpha}(p_0) : = \inf_{\bar \pi \in \bar \Pi} \bar v_{\alpha}^{\bar \pi}(p_0), \qquad p_0 \in \P(\X).$$
Similarly to the average cost case, the correspondence between (DM) and (SM) implies the following:
\begin{itemize}
\item $v_\alpha(x) = \bar v_{\alpha}(\delta_x)$ for all $x \in \X$.
\item $\bar v_{\alpha}$ can be expressed as the result of a partial minimization problem:
\begin{equation} \label{eq-dm-partial-alpha}
    \bar v_\alpha(p_0) = \inf_{(\q_0, \q_1, \ldots) \in \Delta_{p_0}} G_\alpha \big(p_0, \q_0, \q_1, \ldots \big)
\end{equation}
for the function
\begin{equation} \label{eq-G-alpha}
   G_\alpha \big(p_0, \q_0, \q_1, \ldots \big) : = \limsup_{n \to \infty} \textstyle{ \sum_{k=0}^{n} \alpha^k \bar c(\q_k)}.
\end{equation}
\end{itemize}
Then, with arguments almost identical to those given in the preceding subsection for proving Theorem~\ref{thm-ac-basic}, we can draw the following conclusions similar to those derived in the average cost case:
\begin{itemize}
\item [(i)] The function $\bar v_\alpha$ is lower semi-analytic. This implies (as in the proof of Theorem~\ref{thm-ac-basic}(i)) that the optimal value function of (SM), $v_\alpha(x) = \bar v_\alpha(\delta_x)$, is a (finite-valued) lower semi-analytic function.
\item [(ii)] Similarly to the proof of Theorem~\ref{thm-ac-basic}(ii), for each $\epsilon > 0$, we can construct, from an $\epsilon$-optimal solution to (\ref{eq-dm-partial-alpha}), a universally measurable, $\epsilon$-optimal, semi-Markov policy $\pi_\epsilon$ for the original problem (SM). The value function $v^{\pi_\epsilon}_\alpha$ of this policy then satisfies
\begin{equation} \label{eq-prf2a}
 v^{\pi_\epsilon}_\alpha \leq v_\alpha + \epsilon.
\end{equation}
Note that $v^{\pi_\epsilon}_\alpha$ is universally measurable.
\end{itemize}

We now focus on the original problem (SM) and use the above results together with the contraction property of the dynamic programming operator $T_\alpha$ to prove Theorem~\ref{thm-dcoe} for (UC). By Lemma~\ref{lem-uc-Ta-contraction} (see Appendix~\ref{appsec-lem}) and the Banach fixed point theorem, $T_\alpha$ has a unique fixed point in $\A(\X) \cap \M_{\tilde w}(\X)$ for some weight function $\tilde w \geq w$. Since $v_\alpha \in \An(\X) \cap \M_{w}(\X)$ by the preceding proof and $\M_{w}(\X) \subset \M_{\tilde w}(\X)$, we have $v_\alpha \in \An(\X) \cap \M_{\tilde w}(\X)$.  So to prove the theorem, we only need to show $v_\alpha = T_\alpha v_\alpha$. To this end, we prove first $v_\alpha \leq T_\alpha v_\alpha$ and then $v_\alpha \geq T_\alpha v_\alpha$.

Since $v_\alpha \in \An(\X) \cap \M_{w}(\X)$, $T_\alpha v_\alpha \in \An(\X) \cap \M_{w}(\X)$ by Lemma~\ref{lem-T-lsa} and the model condition of (UC). So by a measurable selection theorem \cite[Prop.\ 7.50]{bs}, there exists a universally measurable nonrandomized stationary policy $\mu$ such that
\begin{equation} \label{eq-prf2b}
T_\alpha^\mu v_\alpha \leq T_\alpha v_\alpha + \epsilon,
\end{equation}
where $T_\alpha^\mu: \M(\X) \to \M(\X)$ and it is given by $(T_\alpha^\mu v)(x) = c(x, \mu(x)) + \alpha \int_\X v(y) \, q(dy \,|\, x, a)$ for $v \in \M(\X)$.
Consider the policy $\pi$ that applies $\mu$ at the first stage and then applies $\pi_\epsilon$ afterwards. 
By the monotonicity of $T_\alpha^\mu$ and the inequalities (\ref{eq-prf2a}) and (\ref{eq-prf2b}),
$$  v^\pi_\alpha = T_\alpha^\mu v^{\pi_\epsilon}_\alpha \leq T_\alpha^\mu v_\alpha + \alpha \epsilon \leq  T_\alpha v_\alpha +  2 \epsilon.$$
Since $v^\pi_\alpha \geq v_\alpha$ and $\epsilon$ is arbitrary, we obtain
$v_\alpha \leq T_\alpha v_\alpha.$

For the reverse inequality, consider again the $\epsilon$-optimal semi-Markov policy $\pi_\epsilon$ and express it as $\big(\mu_0(da_0 \mid x_0), \mu_1(da_1 \mid  x_0, x_1), \mu_2(da_2 \mid x_0, x_2), \ldots \big)$. 
For $x \in \X$, since
\begin{align*}
\E^{\pi_\epsilon}_x \left[ \sum_{n=1}^{\infty} \alpha^n c(x_n, a_n) \right] & = \E^{\pi_\epsilon}_x \left[  \E^{\pi_\epsilon}_x \Big[ \sum_{n=1}^{\infty} \alpha^n c(x_n, a_n) \, \Big| \, x_0 , x_1 \Big] \right] 
\geq \E^{\pi_\epsilon}_x \left[ \alpha\, v_\alpha(x_1) \right],
\end{align*}
we have
\begin{align}
  v^{\pi_\epsilon}_\alpha(x) & =  \int_\A c(x,a) \, \mu_0(da \mid x) + \E^{\pi_\epsilon}_x \left[ \sum_{n=1}^{\infty} \alpha^n c(x_n, a_n) \right]  \notag \\
      & \geq  \int_\A  c(x,a)  \, \mu_0(da \mid x) + \alpha \int_\A \int_\X v_\alpha(y) \, q(dy \mid x, a) \, \mu_0(da \mid x) \notag \\
      & \geq \inf_{a \in A(x)} \left\{ c(x, a) +   \alpha \int_\X v_\alpha(y) \, q(dy \mid x, a) \right\} = (T_\alpha v_\alpha)(x).  \label{eq-prf2c}
\end{align}
Combining (\ref{eq-prf2a}) and (\ref{eq-prf2c}) and taking $\epsilon$ to be arbitrarily small, we have
$$v_\alpha + \epsilon \geq v^{\pi_\epsilon}_\alpha \geq T_\alpha v_\alpha, \qquad \Longrightarrow \quad v_\alpha \geq T_\alpha v_\alpha.$$
Hence $v_\alpha = T_\alpha v_\alpha$.

Finally, given $\epsilon' >0$, the nonrandomized stationary policy $\mu$ in (\ref{eq-prf2b}) with $\epsilon = (1 - \alpha) \epsilon'$ is an $\epsilon'$-optimal policy. This follows from combining (\ref{eq-prf2b}) with the monotonicity of $T^\mu_\alpha$ and with the observation that $T^\mu_\alpha$ is a contraction on $\M_{\tilde w}(\X)$ by the same proof of Lemma~\ref{lem-uc-Ta-contraction} and has the value function $v^\mu_\alpha$ as its unique fixed point in $\M_{\tilde w}(\X)$. (We omit the details of the arguments since they are standard and straightforward.) This completes the proof of Theorem~\ref{thm-dcoe} for (UC). \qed

\medskip
\begin{rem}[about the proof] \label{rmk-prf-dcoe-uc} 
We mentioned in the discussion after Theorem~\ref{thm-dcoe} that our proof of its (UC) part is similar to, but does not follow exactly the one given in \cite[Chap.\ 9]{bs} for the case of a bounded one-stage cost function. Let us explain this more here. 

First, there are special cases of the weight function $w(\cdot)$ for which a well-known technique can be applied to convert the problem to one with bounded one-stage costs, and then Theorem~\ref{thm-dcoe}(UC) follows immediately from the result of \cite[Chap.\ 9]{bs}. In particular, suppose that the weight function $w(\cdot)$ is Borel measurable, instead of universally measurable, so that for each $\alpha < 1$, the weight function $\tilde w(\cdot)$ constructed in the proof of Lemma~\ref{lem-uc-Ta-contraction} is also Borel measurable. 
Then using the contraction property of $T_\alpha$ w.r.t.\ the $\| \cdot\|_{\tilde w}$ norm (Lemma~\ref{lem-uc-Ta-contraction}), one can apply Veinott's similarity transformation (\cite{Vei69}; see also \cite[Chap.\ 5.2, p.~100-102]{vdW84}) to convert the $\alpha$-discounted problem into a $\beta$-discounted problem with bounded costs: 
\begin{equation} \label{eq-sim-trans}
   c(x,a) \, \mapsto \, c(x,a)/\tilde w(x), \qquad  \alpha \, q(dy \mid x, a) \, \mapsto \, \alpha \big(\tilde w(y)/\tilde w(x)\big) \, q(dy \,|\, x, a).
\end{equation} 

When $w(\cdot)$ is universally measurable, the conversion just mentioned no longer works. This is because after the similarity transformation (\ref{eq-sim-trans}), the resulting problem has $c(x,a)/\tilde w(x)$ as the one-stage cost function and $\alpha \big(\tilde w(y)/\tilde w(x)\big) q(dy \,|\, x, a)$ as the state transition (sub)-stochastic kernel. The former function is universally measurable and not necessarily lower semi-analytic, and the latter stochastic kernel is universally measurable instead of Borel measurable. So the resulting MDP no longer satisfies the model assumptions in the universal measurability framework for MDPs.

Note also that even when the conversion can be done for each $\alpha$-discounted problem, the model condition of (UC) does not imply that the average cost problem can also be converted to one with bounded one-stage costs.

The proof we gave in this appendix differs from the counterpart in \cite[Chap.\ 9]{bs} in the following way.
Like in the average cost case, we used (DM) to establish $v_\alpha \in \An(\X)$ and the existence of a universally measurable, $\epsilon$-optimal semi-Markov policy in the original MDP. We then combined these facts with the contraction property of $T_\alpha$ to derive the theorem. In \cite[Chap.\ 9]{bs}, the $\alpha$-DCOE for the case of bounded one-stage costs is derived first for (DM) and then transferred to the original problem. This route seems inconvenient for (UC) because with unbounded one-stage costs, the optimal value function in (DM) can take $\pm \infty$ values for some initial distributions $p_0 \in \P(\X)$, although its value at any Dirac measure on $\X$ is finite due to the (UC) model conditions. Thus the dynamic programming operator in (DM) works on the space of extended real-valued, lower semi-analytic functions on $\P(\X)$ and is not a contraction on that space. It is more convenient to work directly with the operator $T_\alpha$ and the original problem, after obtaining the needed optimality properties from (DM). \qed
\end{rem}

\subsection{Three Technical Lemmas} \label{appsec-lem}

The first lemma is about the lower semi-analyticity of $v_\alpha(x)$ as a function of $(\alpha, x)$. We use its (PC) part to prove Lemma~\ref{lem-h-lsa-pc} in Section~\ref{sec-pc-acoi} for establishing the ACOI for the (PC) model.

\begin{lem} \label{lem-valpha-lsc}
{\rm (PC)(UC)} \ On $(0, 1) \times \X$, the function $f(\alpha, x) = v_\alpha(x)$ is lower semi-analytic.
\end{lem}

\begin{proof}
We use the correspondence between the deterministic control model (DM) and the original problem (SM) to prove this lemma.
Recall that in the preceding Section~\ref{appsec-prf-uc}, in the course of proving Theorem~\ref{thm-dcoe} for (UC), 
we have shown the following results regarding the optimal value function $\bar v_\alpha$ in (DM), which hold for (PC) as well by essentially the same arguments:
For $\alpha \in (0,1)$ and $x \in \X$, $\bar v_\alpha(\delta_x) = v_\alpha(x)$; and $\bar v_\alpha$ can be expressed as the result of partial minimization:
\begin{equation} \label{eq-dm-partial-alpha0}
   \bar v_\alpha(p_0) = \inf_{(\q_0, \q_1, \ldots) \in \Delta_{p_0}} G_\alpha \big(p_0, \q_0, \q_1, \ldots \big), \quad \text{where} \ \ \ G_\alpha \big(p_0, \q_0, \q_1, \ldots \big)  = \limsup_{n \to \infty} \sum_{k=0}^{n} \alpha^k \bar c(\q_k). 
\end{equation}   
Since $\alpha \geq 0$ and $\bar c(\cdot)$ is lower semi-analytic (cf.\ Def.~\ref{def-DM}(i)), the product $\alpha^k \bar c(\gamma_k)$ is a lower semi-analytic function of $(\alpha, z)$ on $(0, 1) \times \Delta$ by \cite[Lemma 7.30(4)]{bs}. Therefore, $G_\alpha (z)$ is a lower semi-analytic function of $(\alpha, z)$ on $(0, 1) \times \Delta$ by \cite[Lemma 7.30 (2) and (4)]{bs}. Then, as the result of the partial minimization (\ref{eq-dm-partial-alpha0}), $\bar v_\alpha(p_0)$ is a lower semi-analytic function of $(\alpha, p_0)$ by \cite[Prop.\ 7.47]{bs}. Since $f(\alpha, x) = v_\alpha(x) = \bar v_\alpha(\delta_x)$ and the mapping $(\alpha, x) \mapsto (\alpha, \delta_x)$ is a homeomorphism \cite[Cor.~7.21.1]{bs}, this implies, by \cite[Lemma 7.30(3)]{bs}, that $f$ is a lower semi-analytic function of $(\alpha, x)$.
\end{proof}

The next lemma is about the contraction property of the dynamic programming operator $T_\alpha$ stated in (\ref{eq-contraction-Ta-uc}) for the (UC) model. We use it in the proof of Theorem~\ref{thm-dcoe} for (UC).

\begin{lem} \label{lem-uc-Ta-contraction}
{\rm (UC)} \ For some universally measurable wight function $\tilde w \geq w$, the operator $T_\alpha$ is a contraction on the closed subset $\An(\X) \cap \M_{\tilde w}(\X)$ of the Banach space $(\M_{\tilde w}(\X), \| \cdot\|_{\tilde w})$.
\end{lem}

\begin{proof}
We have already shown that $T_\alpha$ maps $\An(\X)$ into $\An(\X)$ (Lemma~\ref{lem-T-lsa}). To prove the lemma, we need to show that for some $\beta \in (0, 1)$ and $\tilde w \geq w$, $\left\| T_\alpha v - T_\alpha v' \right\|_{\tilde w} \leq \beta \left\| v - v' \right\|_{\tilde w}$ for all $v , v' \in \An(\X) \cap \M_{\tilde w}(\X)$. The proof is essentially the same as the analysis given in \cite[p.\ 45-46]{HL99} (cf.\ the proof of Prop.\ 8.3.4 and Remark 8.3.5 therein). To construct $\tilde w$ with the desired property, we proceed as follows. Define universally measurable functions $c_n$, $n \geq 0$, by
$$ c_0 = w, \qquad c_{n} = \lambda^{n} w + (1 + \lambda + \cdots + \lambda^{n-1}) \, b, \quad n \geq 1.$$
Since $\sup_{a \in A(x)} \int_\X w(y) \, q(dy \,|\, x, a) \leq \lambda w(x) + b$ by the definition of the (UC) model (cf.\ Def.~\ref{def-uc}), 
we have 
\begin{equation} \label{eq-prf-contraction-uc0}
  \sup_{a \in A(x)} \int_\X c_n(y) \, q(dy \mid x, a) \leq c_{n+1}(x), \qquad \forall \, x \in \X.
\end{equation}  
Now choose $\tilde \alpha \in (\alpha, 1)$ and define $\tilde w = \sum_{n=0}^\infty \tilde \alpha^n c_n$. Then $\tilde w \in \M(\X)$ and it is finite-valued because
$$   \tilde w(x) = \sum_{n=0}^\infty \tilde{\alpha}^n c_n(x) \leq \frac{w(x)}{1 - \tilde \alpha \lambda} + \frac{b}{(1 - \tilde \alpha) ( 1- \lambda)} < \infty, \qquad \forall \, x \in \X.$$
Furthermore, for any $x \in \X$ and $a \in A(x)$, by (\ref{eq-prf-contraction-uc0}),
$$    \int_\X \tilde w(y) \, q(dy \mid x, a) = \sum_{n=0}^\infty \tilde{\alpha}^n \int_\X c_n(y) \, q(dy \mid x, a) \leq \sum_{n=0}^\infty \tilde{\alpha}^n c_{n+1}(x) \leq  \tilde w(x) / \tilde{\alpha},$$
from which it follows that 
\begin{equation} \label{eq-prf-contraction-uc1}
   \alpha \cdot \sup_{a \in A(x)} \int_\X \tilde w(y) \, q(dy \mid x, a) \leq (\alpha/\tilde \alpha) \cdot \tilde w(x).
\end{equation}   
Recall also that $\sup_{a \in A(x)} | c(x, a) | \leq \hat c \, w(x) \leq \hat c \, \tilde w(x)$ for all $x \in \X$. This and the relation (\ref{eq-prf-contraction-uc1}) together imply that $T_\alpha$ maps $\An(\X) \cap \M_{\tilde w}(\X)$ into $\An(\X) \cap \M_{\tilde w}(\X)$ and
for $\beta = \alpha/\tilde \alpha < 1$, $T_\alpha$ has the desired contraction property:
$\left\| T_\alpha v - T_\alpha v' \right\|_{\tilde w} \leq \beta \left\| v - v' \right\|_{\tilde w}$ for all $v , v' \in \An(\X) \cap \M_{\tilde w}(\X)$.
\end{proof}

Finally, we prove Lemma~\ref{lem-optpol}. Recall that it is about the existence of optimal and nearly optimal, nonrandomized Markov or stationary policies in (PC) and (UC), as a consequence of the ACOI (\ref{eq-acoi-gen}):
$$  g^* + h(x) \geq \inf_{a \in A(x)} \left\{ c(x, a) + \int_\X h(y) \, q(dy \mid x, a) \right\}, \qquad x \in \X$$
where it is assumed that $g^*$ is constant and $h \in \An(\X)$, both being finite and with $h \geq 0$ for (PC), $\|h\|_w < \infty$ for (UC).
The proof arguments are mostly standard.

\begin{proof}[Proof of Lemma~\ref{lem-optpol}]
For all $x \in \X$, the r.h.s.\ of the ACOI (\ref{eq-acoi-gen}) is greater than $- \infty$ under the assumption on $h$ and the model conditions for (PC) and (UC).
Then, since $h$ is lower semi-analytic and $\Gamma$ is analytic, by a 
measurable selection theorem \cite[Prop.\ 7.50]{bs}, for each $\epsilon > 0$, there exists a universally measurable function $f: \X \to \A$ such that for all $x \in \X$, $f(x) \in A(x)$ and 
\begin{equation} \label{eq-prf-cons-acoi1}
     c(x, f(x)) + \int_\X h(y) \, q(dy \mid x, f(x)) \leq \inf_{a \in A(x)} \left\{ c(x, a) + \int_\X h(y) \, q(dy \mid x, a) \right\} + \epsilon.
\end{equation}     
For $k \geq 0$, let $f_k$ be the function satisfying the above for $\epsilon = \epsilon_k= 2^{-k}$. Then by the ACOI,
\begin{equation} \label{eq-prf-cons-acoi2}
  c(x, f_k(x)) + \int_\X h(y) \, q(dy \mid x, f_k(x)) \leq g^* + h(x) + \epsilon_k, \qquad x \in \X.
\end{equation}  

Consider the nonrandomized Markov policy $\pi = (f_0, f_1, \ldots)$, and let us show that it is average-cost optimal. First, note that for (PC), by iterating the inequality~(\ref{eq-prf-cons-acoi2}) and using the fact $c, h \geq 0$, we have
\begin{equation} \label{eq-prf-cons-acoi3}
     0 \leq \E^\pi_x  \big[ h(x_n) \big] < \infty, \qquad x \in \X, \ n \geq 0.
\end{equation} 
For (UC), the assumption $\| h \|_w < \infty$ and the model condition of (UC) ensure that
\begin{equation} \label{eq-prf-cons-acoi4}
  \E^\pi_x  \big[ | h(x_n) | \big] \leq \| h\|_w \cdot \E^\pi_x  \big[ w(x_n) \big] \leq \| h\|_w \cdot \big(\lambda^n w(x) + b/(1-\lambda)\big), \qquad x \in \X, \ n \geq 0.
\end{equation}
Now for both (PC) and (UC), for all $x \in \X$, we have that
\begin{align}
    J_{n+1}(\pi, x) & = \E^\pi_x \Big[ \sum_{k=0}^{n} c(x_k, f_k(x_k)) \Big] 
        =  \sum_{k=0}^{n} \E^\pi_x \Big[ c(x_k, f_k(x_k)) + h(x_{k+1}) - h(x_k) \Big] + h(x) - \E^\pi_x \big[ h(x_{n+1}) \big] \notag \\
       & \leq (n+1) \, g^* + h(x) - \E^\pi_x \big[ h(x_{n+1}) \big] + \sum_{k=0}^{n} \epsilon_k,  \label{eq-prf-cons-acoi5}
\end{align}    
where the second equality is valid in view of (\ref{eq-prf-cons-acoi3}) and (\ref{eq-prf-cons-acoi4}) (since they rule out the occurrence of $+\infty - \infty$ or $-\infty + \infty$ when we add or subtract the expectation of $h(x_k)$), and the last inequality follows from~(\ref{eq-prf-cons-acoi2}).   
Together with (\ref{eq-prf-cons-acoi3}) and (\ref{eq-prf-cons-acoi4}), (\ref{eq-prf-cons-acoi5})  
implies that in (PC) and (UC), 
$$\limsup_{n \to \infty} J_{n}(\pi, x)/n \leq g^*, \qquad \forall \, x \in \X.$$
Hence $\pi$ must be average-cost optimal. 

The same reasoning shows that for each $\epsilon > 0$, with $f$ as given in (\ref{eq-prf-cons-acoi1}), the nonrandomized stationary policy $\pi = (f, f, \ldots)$ is $\epsilon$-optimal.
When the infimum in the r.h.s.\ of the ACOI is attained for every $x \in \X$,  by the same selection theorem \cite[Prop.\ 7.50(b)]{bs}, the universally measurable function $f$ in (\ref{eq-prf-cons-acoi1}) can be chosen to satisfy the equality for $\epsilon = 0$. Then by the same reasoning as above, the nonrandomized stationary policy $\pi = (f, f, \ldots)$ is average-cost optimal.
\end{proof}

\end{appendices}

\addcontentsline{toc}{section}{References} 
\bibliographystyle{apa} 
\let\oldbibliography\thebibliography
\renewcommand{\thebibliography}[1]{%
  \oldbibliography{#1}%
  \setlength{\itemsep}{0pt}%
}
{\fontsize{9}{11} \selectfont
\bibliography{ac_BorelDP_bib}}

\end{document}